\documentclass[preprint]{elsarticle}

\usepackage{graphicx}
\usepackage{amssymb}
\usepackage{lineno}
\usepackage{hyperref}
\usepackage{amsmath}
\usepackage{amsthm}
\usepackage{stmaryrd}
\usepackage{multirow}
\usepackage{enumitem}
\usepackage{dsfont}
\usepackage{makecell}
\usepackage{soul}
\usepackage{appendix}
\usepackage{multirow}
\usepackage[dvipsnames]{xcolor}
\usepackage{arydshln}
\setuldepth{Paris}

\newtheoremstyle{theorems}{1em}{1em}{}{}{\bfseries}{}{ }{}
\theoremstyle{theorems}
\newtheorem{theorem}{Theorem}[section]

\newtheorem{definition}{Definition}[section]
\newtheorem{lemma}{Lemma}[section]
\newtheorem{remark}{Remark}[section]

\newtheorem{principle}{Principle}[section]
\newtheorem{corollary}{Corollary}[section]

\definecolor{lightblue}{rgb}{0.8,0.8,1}
\sethlcolor{lightblue}

\newcommand{\R}{\mathbb{R}}

\newcommand{\M}{\mathcal{M}}
\newcommand{\mc}{\mathcal}

\newcommand{\lto}{\longrightarrow}
\newcommand{\lmto}{\longmapsto}

\newcommand{\ls}{\leqslant}
\newcommand{\gs}{\geqslant}

\newcommand{\tr}{\mathrm{tr}}

\newcommand{\inv}{\mathrm{inv}}

\newcommand{\pow}{\mathrm{pow}}
\newcommand{\diag}{\mathrm{diag}}
\newcommand{\Diag}{\mathrm{Diag}}

\newcommand{\adj}{\mathrm{adj}}

\newcommand{\Id}{\mathrm{Id}}

\newcommand{\Sym}{\mathrm{Sym}}

\newcommand{\Orth}{\mathrm{O}}

\newcommand{\SPD}{\mathrm{SPD}}

\newcommand{\Log}{\mathrm{Log}}

\newcommand{\dotprod}[2]{\langle #1|#2\rangle}

\newcommand{\Frob}{\mathrm{Frob}}

\newcommand{\ME}{\mathrm{ME}}

\newcommand{\MPE}{\mathrm{MPE}}

\newcommand{\BKM}{\mathrm{BKM}}

\newcommand{\sys}[4]{
\left\{
\begin{array}{ccc}
#1 & \mathrm{if} & #2\\ \relax
#3 & \mathrm{if} & #4\\
\end{array}
\right\}
}

\newcommand{\syst}[8]{
\left\{
\begin{array}{ccc}
#1 & \mathrm{if} & #2\\ \relax
#3 & \mathrm{if} & #4\\ \relax
#5 & \mathrm{if} & #6\\ \relax
#7 & \mathrm{if} & #8\\
\end{array}
\right\}
}

\journal{Differential Geometry and its Applications}

\begin{document}

\begin{frontmatter}

\title{The geometry of mixed-Euclidean metrics\\on symmetric positive definite matrices}
\author{Yann Thanwerdas\corref{cor1}}
\ead{yann.thanwerdas@inria.fr}
\author{Xavier Pennec}
\ead{xavier.pennec@inria.fr}

\cortext[cor1]{Corresponding author}

\address{Université Côte d'Azur and Inria, Epione Project Team, Sophia Antipolis\\2004 route des Lucioles, 06902 Valbonne Cedex, France
\vspace{-10mm}
}

\begin{abstract}
Several Riemannian metrics and families of Riemannian metrics were defined on the manifold of Symmetric Positive Definite (SPD) matrices. Firstly, we formalize a common general process to define families of metrics: the principle of deformed metrics. We relate the recently introduced family of alpha-Procrustes metrics to the general class of mean kernel metrics by providing a sufficient condition under which elements of the former belongs to the latter. 
Secondly, we focus on the principle of balanced bilinear forms that we recently introduced. We give a new sufficient condition under which the balanced bilinear form is a metric. It allows us to introduce the Mixed-Euclidean (ME) metrics which generalize the Mixed-Power-Euclidean (MPE) metrics. We unveal their link with the $(u,v)$-divergences and the $(\alpha,\beta)$-divergences of information geometry and we provide an explicit formula of the Riemann curvature tensor. We show that the sectional curvature of all ME metrics can take negative values and we show experimentally that the sectional curvature of all MPE metrics but the log-Euclidean, power-Euclidean and power-affine metrics can take positive values.
\end{abstract}


\end{frontmatter}


\section{Introduction}

The convex cone of Symmetric Positive Definite (SPD) matrices is a manifold on which several Riemannian metrics were defined: Euclidean, Fisher-Rao/affine-invariant \cite{Skovgaard84,Amari00,Moakher05,Pennec06,Lenglet06-JMIV,Fletcher07}, log-Euclidean \cite{Arsigny06}, Bures-Wasserstein \cite{Dowson82,Olkin82,Dryden09,Takatsu10,Takatsu11,Bhatia19}, Bogoliubov-Kubo-Mori \cite{Petz93,Michor00}, log-Cholesky \cite{Lin19}... Several families of metrics encompassing them were defined to understand their common properties, their differences and the level of generality of each property: kernel metrics and mean kernel metrics \cite{Hiai09,Hiai12}, power-Euclidean \cite{Dryden10}, alpha-Procrustes \cite{HaQuang19}, deformed-affine \cite{Thanwerdas19-affine}, mixed-power-Euclidean \cite{Thanwerdas19-balanced}, extended kernel metrics, bivariate separable metrics \cite{Thanwerdas21-LAA}... In particular, kernel metrics form a very general family of $\Orth(n)$-invariant metrics indexed by kernel maps $\phi:(0,\infty)^2\lto(0,\infty)$ acting on the eigenvalues of SPD matrices. This family contains many $\Orth(n)$-invariant metrics and it has good stability properties. The subclass of mean kernel metrics, for which the kernel maps have monotonicity properties, is interesting because it provides a necessary and sufficient condition for geodesic completeness. Hence, kernel metrics and mean kernel metrics appear as sufficiently general families with interesting properties so it is a natural framework to work in. However, this class contains metrics with very different geometries so it motivates us to define subfamilies of metrics which share more geometric properties with one another.

In previous works, we introduced two principles for building families of Riemannian metrics that share interesting properties: the principle of deformed metrics \cite{Thanwerdas19-affine} and the principle of balanced bilinear forms \cite{Thanwerdas19-balanced}. Deforming metrics (or datasets of SPD matrices) via a diffeomorphism is a very common procedure to define families of metrics. In particular, kernel metrics are stable by univariate diffeomorphisms, those which are characterized by their action on eigenvalues. However, mean kernel metrics are not stable by all univariate diffeomorphisms because of the monotonicity requirement. In this work, we gather many constructions of deformed metrics and we contribute a sufficient condition under which alpha-Procrustes metrics are mean kernel metrics.

The balanced bilinear form of two flat metrics is defined by composing the Frobenius inner product with the parallel transport of each flat metric \cite{Thanwerdas19-balanced}. When the bilinear form is a metric, it forms a dually-flat manifold along with the two flat Levi-Civita connections of the flat metrics. In the case where the two flat metrics are power-Euclidean metrics, the balanced bilinear form is a metric called the mixed-power-Euclidean metric. In this work, we give a new sufficient condition for a balanced bilinear form to be a metric, namely that the flat metrics are univariately-deformed-Euclidean metrics, which allows to define the new family of Mixed-Euclidean metrics. Then, we provide the geometric operations of Mixed-Euclidean metrics regarding information geometry and Riemannian geometry. In particular, our main contributions are on the one hand the link we establish between Mixed-Euclidean/Mixed-Power-Euclidean metrics and the $(u,v)/(\alpha,\beta)$-divergences of information geometry, and on the other hand the expression of the Riemann curvature tensor of Mixed-Euclidean metrics.

In Section 2, we present our notations and the preliminary concepts of univariate maps and kernel metrics. In Section 3, we study deformed metrics and we relate the family of alpha-Procrustes metrics to the class of mean kernel metrics. In Section 4, we recall the main concepts of information geometry, we state the principle of balanced bilinear forms and we explain the relation between the two. In Section 5, we introduce the new family of Mixed-Euclidean metrics and we study its geometry. We conclude and discuss some perspectives in Section 6. The proofs of the results are presented in appendix.

\section{Notations and preliminary concepts}

In this section, we introduce some notations and we recall two concepts that are used throughout the paper. The first one is the concept of univariate map on SPD matrices: it is a map acting on the eigenvalues, such as the symmetric matrix logarithm or the power maps. The successive differentials of smooth univariate maps can be expressed in closed form modulo eigenvalue decomposition thanks to the functions called divided differences \cite{Bhatia97}. This main advantage explains why they are ubiquitous as indexing collections of families of metrics. Secondly, we recall the main facts about classes of kernel and mean kernel metrics introduced in \cite{Hiai09,Hiai12}.

\subsection{Notations}

We denote $\Sym(n)$ the vector space of real symmetric matrices of size $n$, $\SPD(n)$ the manifold of SPD matrices, $\Orth(n)$ the orthogonal group, $\Diag^+(n)$ the group of positive diagonal matrices.

On the manifold $\SPD(n)$, we denote $T_\Sigma\SPD(n)$ the tangent space at $\Sigma\in\SPD(n)$. Given a metric $g^{\mathrm{I}}$ on the manifold $\SPD(n)$ where $\mathrm{I}$ is any index characterizing the metric, we denote $\nabla^{\mathrm{I}}$ its Levi-Civita connection, $R^{\mathrm{I}}$ the Riemann curvature tensor, $T^{\mathrm{I}}$ the torsion tensor, $\Pi^{\mathrm{I}}$ the parallel transport. We omit the index when the context is clear.

Given a matrix $M$, we denote $M_{ij}$ or $[M]_{ij}$ the $(i,j)$-th coefficient of $M$. Given coefficients $(M_{ij})_{1\ls i,j\ls n}\in\R^{n^2}$, we denote $[M_{ij}]_{i,j}$ the matrix with $(i,j)$-th entry $M_{ij}$. Given $(d_1,...,d_n)\in\R^n$, we denote $\diag(d_1,...,d_n)$ the corresponding diagonal matrix.

We recall that $\exp:\Sigma\in\Sym(n)\lmto\sum_{k=0}^{+\infty}{\frac{1}{k!}\Sigma^k}\in\SPD(n)$ is a diffeomorphism whose inverse is the symmetric matrix logarithm denoted $\log:\SPD(n)\lto\Sym(n)$.

\subsection{Univariate maps}

In this paper, we call $\Orth(n)$-equivariant map a map $f:\SPD(n)\lto\Sym(n)$ such that $f(R\Sigma R^\top)=R\,f(\Sigma)R^\top$ for all $\Sigma\in\SPD(n)$ and $R\in\Orth(n)$. Among $\Orth(n)$-equivariant maps, we focus on the class of univariate maps.

\begin{definition}[Univariate maps]
A univariate map is an $\Orth(n)$-equivariant map $f:\SPD(n)\lto\Sym(n)$ such that there exists a map on positive real numbers also denoted $f:(0,\infty)\lto\R$ such that $f(PDP^\top)=P\,\Diag(f(d_1),...,f(d_n))\,P^\top$ for all $P\in\Orth(n)$ and $D\in\Diag^+(n)$ with $D=\Diag(d_1,...,d_n)$.
\end{definition}

Any $f:(0,\infty)\lto\R$ can be extended into a univariate map and if the former is of class $\mc{C}^1$ (resp. $\mc{C}^2$, resp. a $\mc{C}^1$-diffeomorphism), then the latter is differentiable (resp. two times differentiable, resp. a diffeomorphism) \cite{Bhatia97,Thanwerdas21-LAA}. We denote $\mc{U}niv$ the set of smooth univariate diffeomorphisms. In addition, the differential and the Hessian of a smooth univariate map can be expressed thanks to the first and second divided differences as follows.

\begin{definition}[Divided differences]\label{def_first_divided_difference}\cite{Bhatia97}
\begin{enumerate}
    \item Let $f\in\mc{C}^1(\R,\R)$. The first divided difference of $f$ is the continuous symmetric map $f^{[1]}:\R^2\lto\R$ defined for $x,y\in\R$ by:
    \begin{equation}
        f^{[1]}(x,y)=\sys{\frac{f(x)-f(y)}{x-y}}{x\ne y}{f'(x)}{x=y}.
    \end{equation}
    
    \item Let $f\in\mc{C}^2(\R,\R)$. The second divided difference of $f$ is the continuous symmetric map $f^{[2]}:\R^3\lto\R$ defined for $x,y,z\in\R$ by:
    \begin{equation}
        f^{[2]}(x,y,z)=\syst{(f^{[1]}(x,\cdot))^{[1]}(y,z)=\frac{f^{[1]}(x,z)-f^{[1]}(x,y)}{z-y}}{y\ne z}{(f^{[1]}(y,\cdot))^{[1]}(z,x)=\frac{f^{[1]}(y,x)-f^{[1]}(y,z)}{x-z}}{z\ne x}{(f^{[1]}(z,\cdot))^{[1]}(x,y)=\frac{f^{[1]}(z,y)-f^{[1]}(z,x)}{y-x}}{x\ne y}{\frac{1}{2}f''(x)}{x=y=z}.
    \end{equation}
\end{enumerate}
\end{definition}

If $f\in\mc{C}^2(\R,\R)$, then one can check that the differential of $f^{[1]}$ at $(x,y)\in\R^2$ is:
\begin{equation}
    d_{(x,y)}f^{[1]}(h,k)=\sys{\frac{f'(x)h-f'(y)k}{x-y}+\frac{f(x)-f(y)}{(x-y)^2}(h-k)}{x\ne y}{\frac{f''(x)}{2}(h+k)}{x=y},
\end{equation}
so one can prove that $df^{[1]}$ is continuous and $f^{[1]}\in\mc{C}^1(\R^2,\R)$. This also proves that $\frac{\partial f^{[1]}}{\partial x}(x,x)=\frac{f''(x)}{2}$ and that $f^{[2]}$ is continuous.

From now on, all univariate maps are assumed to be smooth.

\begin{lemma}[Differential and Hessian of a univariate map]\label{lemma_diff_univariate}\cite{Bhatia97}
The differential and the Hessian of a univariate map $f$ are $\Orth(n)$-equivariant: $d_{PDP^\top}f(PXP^\top)=P\,d_Df(X)\,P^\top$ and $H_{PDP^\top}f(PXP^\top,PYP^\top)=P\,H_Df(X,Y)\,P^\top$. Hence, they are determined by their values at diagonal matrices $D\in\Diag^+(n)$, which are given by the following formulae:
\begin{align}
    [d_Df(X)]_{ij}&=f^{[1]}(d_i,d_j)X_{ij},\\
    [H_Df(X,X)]_{ij}&=2\sum_{k=1}^n{f^{[2]}(d_i,d_j,d_k)X_{ik}X_{jk}}.
\end{align}
\end{lemma}

\subsection{Classes of $\Orth(n)$-invariant metrics}

The class of kernel metrics is a subclass of $\Orth(n)$-invariant metrics on SPD matrices indexed by smooth bivariate symmetric maps $\phi:(0,\infty)^2\lto(0,\infty)$ \cite{Hiai09}. The advantages of this class are the simple formulation of its elements, some important results on the metrics (completeness, cometric) and some important stability properties of the class. Hence, it is a good ambient class to define subfamilies of metrics. Therefore in this section, we recall the definition of kernel metrics, the refinement of mean kernel metrics and the results on completeness, cometric and stability.

\subsubsection{Kernel metrics}

\begin{definition}[Kernel metric]
A \textit{kernel metric} \cite{Hiai09} is an $\Orth(n)$-invariant metric for which there is a smooth bivariate map $\phi:(0,\infty)^2\lto(0,\infty)$ such that $g_\Sigma(X,X)=g_D(X',X')=\sum_{i,j}{\frac{1}{\phi(d_i,d_j)}X_{ij}'^2}$, where $X=PX'P^\top$, $\Sigma=PDP^\top$ with $P\in\Orth(n)$ and $D=\Diag(d_1,...,d_n)$.
\end{definition}

Important examples of kernel metrics are the Euclidean, the log-Euclidean, the affine-invariant, the Bures-Wasserstein and the Bogoliubov-Kubo-Mori metrics:
\begin{align}
    &(\mathrm{Euclidean}) & g^{\mathrm{E}}_\Sigma(X,X)&=\tr(X^2),\label{eq:euclidean_metric}\\
    &(\mathrm{Log}\text{-}\mathrm{Euclidean}) &g^{\mathrm{LE}}_\Sigma(X,X)&=\tr(d_\Sigma\log(X)^2),\label{eq:log-euclidean_metric}\\
    &(\mathrm{Affine}\text{-}\mathrm{invariant}) &g^{\mathrm{A}}_\Sigma(X,X)&=\tr((\Sigma^{-1}X)^2),\label{eq:affine-invariant_metric}\\
    &(\mathrm{Bures}\text{-}\mathrm{Wasserstein}) &g^{\mathrm{BW}}_\Sigma(X,X)&=\tr(\Sigma\mc{S}_\Sigma(X)^2),\label{eq:bures-wasserstein_metric}\\
    &(\mathrm{Bogoliubov}\text{-}\mathrm{Kubo}\text{-}\mathrm{Mori}) &g^{\mathrm{BKM}}_\Sigma(X,X)&=\tr(d_\Sigma\log(X)X),\label{eq:bkm_metric}
\end{align}
where $\mc{S}_\Sigma(X)$ denotes the solution of the Sylvester equation $X=\Sigma\mc{S}_\Sigma(X)+\mc{S}_\Sigma(X)\Sigma$. A review of the definitions, geometric properties and main references on these five metrics can be found in \cite{Thanwerdas21-LAA}.

\subsubsection{The refinement of mean kernel metrics}

There is a refinement of kernel metrics where the bivariate function $\phi$ relies on a function called a symmetric homogeneous mean \cite{Hiai09}. These subclasses provide a nice necessary and sufficient condition for geodesic completeness.

\begin{definition}[Mean kernel metrics]\label{def:mean_kernel_metrics}\cite{Hiai09}
A \textit{mean kernel metric} is a kernel metric characterized by a bivariate map $\phi$ of the form $\phi(x,y)=a\,m(x,y)^\theta$ where $a>0$ is a positive coefficient, $\theta\in\R$ is a homogeneity power and $m:(0,\infty)^2\lto(0,\infty)$ is a symmetric homogeneous mean, that is:
\begin{enumerate}
    \itemsep0em
    \item symmetric, i.e. $m(x,y)=m(y,x)$ for all $x,y>0$,
    \item homogeneous, i.e. $m(cx,cy)=c\,m(x,y)$ for all $c,x,y>0$,
    \item non-decreasing in both variables,
    \item $\min(x,y)\ls m(x,y)\ls\max(x,y)$ for all $x,y>0$. It implies $m(x,x)=x$.
\end{enumerate}
\end{definition}

\begin{theorem}[Completeness of mean kernel metrics]\cite{Hiai09}
Mean kernel metrics are geodesically complete if and only if $\theta=2$.
\end{theorem}

\subsubsection{Main results}

The five kernel metrics cited above are mean kernel metrics. The mean functions are summarized in Table \ref{tab:mean_kernel_metrics}.
    
\begin{table}[htbp]
    \centering
    \begin{tabular}{|c|c|c|c|c|}
    \hline
    Metric & Kernel $\phi(x,y)$ & Mean $m$ & Power $\theta$\\
    \hline
    Euclidean & $1$ & Any mean & $0$ \\
    Log-Euclidean & $(\frac{x-y}{\log(x)-\log(y)})^2$ & Logarithmic mean & $2$\\
    Affine-invariant & $xy$ & Geometric mean & $2$\\
    Bures-Wasserstein & $4\,\frac{x+y}{2}$ & Arithmetic mean & $1$ \\
    Bogoliubov-Kubo-Mori & $\frac{x-y}{\log(x)-\log(y)}$ & Logarithmic mean & $1$\\
    \hline
    \end{tabular}
    \caption{Bivariate functions of the main $\Orth(n)$-invariant metrics on SPD matrices.}
    \label{tab:mean_kernel_metrics}
\end{table}
    
Moreover, the class of kernel metrics is stable under pullback by univariate diffeomorphisms \cite{Hiai09}. However, the class of \textit{mean} kernel metrics is not stable under univariate diffeomorphisms, essentially because of the third condition of a symmetric homogeneous mean. Indeed, the mean has to be non-decreasing in both variables which is neither a differential nor a Riemannian property.

In addition, the class of kernel metrics is cometric stable \cite{Thanwerdas21-LAA}. Indeed, the cometric is a metric on the cotangent bundle $T^*\SPD(n)\simeq\SPD(n)\times\Sym(n)^*$. Thanks to the Riesz theorem, the Frobenius inner product provides the identification $\Sym(n)^*\simeq\Sym(n)$ so the cometric can be considered as a metric. The cometric of the kernel metric characterized by $\phi$ is the kernel metric characterized by $\frac{1}{\phi}$.

\section{Deformed metrics}

Log-Euclidean metrics on SPD matrices are pullback metrics of Euclidean metrics on the vector space of symmetric matrices via the symmetric matrix logarithm $\log:\SPD(n)\lto\Sym(n)$. This geometric construction of a metric on SPD matrices based on a diffeomorphism $f$ is commonly used to define families of metrics on SPD matrices indexed by automorphisms of SPD matrices. Indeed, even if these metrics are isometric, they do not give the same results in data analyses. It is actually equivalent to compute with the metric $g$ on the transformed dataset $[f(\Sigma_1),...,f(\Sigma_N)]$ or to compute with the pullback metric $f^*g$ on the initial dataset $[\Sigma_1,...,\Sigma_N]$.

In this section, we give examples of situations in the literature where such transformations are applied to the data (section \ref{subsec:use_of_a_deformation_in_the_literature}), then we unify them into our principle of deformed metrics and we give the fundamental Riemannian operations (distance, geodesics, curvature, parallel transport) of the deformed metrics (section \ref{subsec:principle_of_deformed_metrics}). In Section \ref{subsec:deformed_bures-wasserstein}, we contribute the new family of deformed-Wasserstein metrics based on this principle which comprises the family of alpha-Procrustes metrics \cite{HaQuang19}. We also give a sufficient condition under which alpha-Procrustes metrics are mean kernel metrics.

\subsection{Use of a deformation in the literature}\label{subsec:use_of_a_deformation_in_the_literature}

As mentioned before, the class of kernel metrics is stable by pullback under univariate diffeomorphisms \cite{Hiai09,Thanwerdas21-LAA}. In particular, pullbacks of the Euclidean metric and the affine-invariant metric under power diffeomorphisms are detailed in the original paper on kernel metrics \cite{Hiai09}. They were later called power-Euclidean \cite{Dryden10} and power-affine metrics, or more generally deformed-Euclidean and deformed-affine metrics for an arbitrary diffeomorphism \cite{Thanwerdas19-affine}. Moreover, power-Euclidean metrics are \textit{mean} kernel metrics for any power and power-affine metrics are \textit{mean} kernel metrics if and only if the power belongs to $[-2,2]$ \cite{Hiai09}. Since power-Euclidean metrics interpolate between the log-Euclidean, the Wigner-Yanase/square-root and the Euclidean metrics, an optimization procedure was proposed on the parameter to choose the most appropriate metric on a dataset of covariance matrices for Diffusion Tensor Imaging (DTI) \cite{Dryden10}. It is common in DTI to compute with precision matrices which are the inverses of covariance matrices, $\inv(\Sigma)=\Sigma^{-1}$ \cite{Lenglet04}, or with other transformations of the covariance matrices such as the adjugate function $\adj(\Sigma)=\det(\Sigma)\Sigma^{-1}$ \cite{Fuster16}. More recently, the family of alpha-Procrustes metrics was introduced by pullback under power diffeomorphisms of the Bures-Wasserstein metric, as for power-Euclidean and power-affine metrics, and it was extended to the infinite dimension in the context of Reproducing Kernel Hilbert Spaces (RKHS) \cite{HaQuang19}.

In papers where the power diffeomorphisms are used to define power-Euclidean, power-affine and alpha-Procrustes metrics \cite{Hiai09,Dryden10,Hiai12,Thanwerdas19-affine,HaQuang19}, it is often noticed that the limit when the power tends to 0 is the log-Euclidean metric. This is actually a general fact. Indeed, from a Riemannian metric $g$, it is possible to construct a one-parameter family of metrics $(g^{(p)})_{p\in\R^*}$ by taking the pullback by the power diffeomorphism $\pow_p:\Sigma\in\SPD(n)\lmto\Sigma^p\in\SPD(n)$ for $p\ne 0$ and to scale it by $\frac{1}{p^2}$, that is $g^{(p)}_\Sigma(X,X)=\frac{1}{p^2}g_{\Sigma^p}(d_\Sigma\pow_p(X),d_\Sigma\pow_p(X))$ for all $\Sigma\in\SPD(n)$ and all $X\in T_\Sigma\SPD(n)$. Then when $p$ tends to 0, $g^{(p)}$ tends to the log-Euclidean metric associated to the inner product $g_{I_n}$, that is $g^{(p)}_\Sigma(X,X)\underset{p\to 0}{\lto}g_{I_n}(d_\Sigma\log(X),d_\Sigma\log(X))$ for all $\Sigma\in\SPD(n)$ and all $X\in T_\Sigma\SPD(n)$.

\subsection{Principle of deformed metrics}\label{subsec:principle_of_deformed_metrics}

\begin{principle}[Principle of deformed metrics]
Let $g$ be a Riemannian metric on $\SPD(n)$ and $f:\SPD(n)\lto\SPD(n)$ be a diffeomorphism. Then the $f$-deformed metric is defined as the pullback metric $f^*g$. It is a Riemannian metric on $\SPD(n)$ which is isometric to $g$ and whose expression is:
\begin{equation}
    (f^*g)_\Sigma(X,X)=g_{f(\Sigma)}(d_\Sigma f(X),d_\Sigma f(X)).
\end{equation}
\end{principle}

All the Riemannian operations of a deformed metric are obtained by pulling back the formulae that are known for the initial metric, as shown in Table \ref{tab:riemannian_operations_deformed}.

\begin{table}[htbp]
\centering
\begin{tabular}{|c|c|}
    \hline
    Metric & $g^f_\Sigma(X,X)=g_{f(\Sigma)}(d_\Sigma f(X),d_\Sigma f(X))$ \\
    \hline
    Distance & $d^f(\Sigma,\Lambda)=d(f(\Sigma),f(\Lambda))$\\
    \hline
    Levi-Civita & $d_\Sigma f(\nabla^f_{X_\Sigma}Y)=\nabla_{d_\Sigma f(X)}(df(Y))$\\
    \hline
    Curvature & $d_\Sigma f(R^f_\Sigma(X,Y)Z)=R_{f(\Sigma)}(d_\Sigma f(X),d_\Sigma f(Y))d_\Sigma f(Z)$\\
    \hline
    Geodesics & $f(\gamma^f_{(\Sigma,X)}(t))=\gamma_{(f(\Sigma),d_\Sigma f(X))}(t)$\\
    \hline
    Logarithm & $d_\Sigma f(\Log^f_\Sigma(\Lambda))=\Log_{f(\Sigma)}(f(\Lambda))$\\
    \hline
    Parallel transport & $d_\Lambda f(\Pi^f_{\gamma;\Sigma\to\Lambda}X)=\Pi_{f\circ\gamma;f(\Sigma)\to f(\Lambda)}(d_\Sigma f(X))$\\
    \hline
\end{tabular}
\caption{Riemannian operations of deformed metrics on SPD matrices}
\label{tab:riemannian_operations_deformed}
\end{table}

A fundamental stability property is that if $g$ is $\Orth(n)$-invariant and if $f$ is $\Orth(n)$-equivariant, then the deformed metric $f^*g$ is also $\Orth(n)$-invariant. Moreover, as mentioned before, if $g$ is a kernel metric and if $f$ is univariate, then the deformed metric $f^*g$ is a kernel metric and the set $\{f^*g,f\in\mc{U}niv\}$ forms a family of kernel metrics that is closed under pullback by univariate diffeomorphisms \cite{Hiai09}. The Riemannian operations that are known in closed form for $g$ are also known in closed form for $f^*g$.

\subsection{The new family of deformed-Wasserstein metrics}\label{subsec:deformed_bures-wasserstein}

\begin{definition}[Deformed-Wasserstein metrics]
A deformed-Wasserstein metric is the pullback metric by a univariate diffeomorphism of the Bures-Wasserstein metric (Formula (\ref{eq:bures-wasserstein_metric})).
\end{definition}

The family of deformed-Wasserstein metrics contains the family of alpha-Procrustes metrics since they are pullbacks of the Bures-Wasserstein metric by the power diffeomorphism $\pow_{2\alpha}$ scaled by $\frac{1}{4\alpha^2}$ \cite{HaQuang19}. In this work, we designate alpha-Procrustes metrics as power-Wasserstein metrics to be consistent with power-Euclidean and power-affine metrics and to parameterize the family by $p\in\R^*$, the correspondence being $p=2\alpha$. As argued earlier, we can say that the log-Euclidean metric belongs to deformed-Wasserstein metrics so we can designate it as power-Wasserstein with power $p=0$.

Since the Bures-Wasserstein metric is a mean kernel metric, it is tempting to determine when a power-Wasserstein metric is a mean kernel metric, in analogy to the work done for the power-Euclidean and the power-affine metrics \cite{Hiai09}. Here we give a sufficient condition under which a power-Wasserstein metric is a mean kernel metric. The proof is in Appendix \ref{app}.

\begin{theorem}[Sufficient condition for power-Wasserstein to be mean kernel]\label{thm:power-wasserstein}
The power-Wasserstein metric of parameter $p\ls 1$ is a mean kernel metric.
\end{theorem}

This condition does not seem to be sufficient. Indeed, after numerical simulations, we conjecture that there exists $p_0\in(2.61,2.611)$ such that the power-Wasserstein metric of parameter $p$ is a mean kernel metric if and only if $p\in(-\infty,1]\cup[p_0,+\infty)$. Moreover, the proof actually tells that if $p\in(1,2]$, then it is not a mean kernel metric.

In this section, we gathered the deformations of Riemannian metrics or of SPD datasets under our principle of deformed metrics. Therefore from a metric we can define the family of power deformations of this metric, which tends to a log-Euclidean metric when the power tends to 0. Moreover, the family of univariate deformations of a kernel metric is a stable subfamily of kernel metrics and it is interesting to determine when these metrics are mean kernel metrics. It seems to be a quite difficult problem for general univariate deformations. On the example of power deformations of the Bures-Wasserstein metric, we gave the sufficient condition $p\ls 1$. To the best of our knowledge, determining necessary and sufficient conditions for deformed-Euclidean, deformed-affine and deformed-Wasserstein metrics to be mean kernel metrics remains an open problem.

\section{Balanced metrics}

The affine-invariant and the Bogoliubov-Kubo-Mori metrics were shown to provide a dually-flat structure, that is a couple of flat affine connections which are dual with respect to the metric. This is a rich geometric structure which provides a so called canonical divergence, a potentials and specific algorithms \cite{Amari00,Banerjee05,Nielsen09}. A dually-flat manifold is a Hessian manifold where the potential is defined globally \cite{Shima97,Amari14-DGA}. Inspired by the characterization of the duality based on parallel transport, we introduced a preliminary version of the principle of balanced bilinear forms in \cite{Thanwerdas19-balanced} which allows to define a bilinear form $g^0$ on SPD matrices from two flat Riemannian metrics $g,g^*$ by $g^0_\Sigma(X,Y)=\tr((\Pi_{\Sigma\to I_n} X)(\Pi^*_{\Sigma\to I_n}Y))$ where $\Pi,\Pi^*$ denote the respective parallel transports. The term ``balanced" was chosen because the bilinear form relies half on each of the two flat metrics. We showed that if the two flat metrics are power-Euclidean metrics, then the balanced bilinear form is symmetric and positive definite, i.e. a metric. In this section, we give a weaker condition under which the bilinear form is a metric.

To ease the comprehension of this section, we recall the main concepts of information geometry, especially dually-flat manifolds and related notions, in Section \ref{subsec:information_geometry_and_dually-flat_manifolds}. In Section \ref{subsec:principle_of_balanced_bilinear_forms}, we provide a new condition on the two flat metrics so that the balanced bilinear form is a metric: it is sufficient to assume that the flat metrics are univariately-deformed-Euclidean metrics.

\subsection{Information geometry and dually-flat manifolds}\label{subsec:information_geometry_and_dually-flat_manifolds}

Before introducing the specific concepts of information geometry, we recall the definition of an affine map between manifolds equipped with affine connections and the definition of a flat affine connection. We denote $\partial$ the canonical affine connection on a vector space.

\begin{definition}[Affine map]
Let $\mc{M},\mc{M}'$ be two manifolds with respective affine connections $\nabla,\nabla'$. We say that $f:\mc{M}\lto\mc{M}'$ is an affine map if for all vector fields $X,Y$ on $\mc{M}$, we have $\nabla'_{f_*(X)}f_*(Y)=f_*(\nabla_XY)$.
\end{definition}

\begin{definition}[Flat affine connection]\label{def:flat_affine_connection}
Depending on domains of research and authors, a flat affine connection is a connection such that:
\begin{enumerate}
    \itemsep0em
    \item (Affine geometry) $R=0$ and $T=0$, \label{enum:flat2}
    \item (Information geometry) \cite[Section 1.7]{Amari00} there exists a global chart $f:(\mc{M},\nabla)\lto(\R^{\dim\mc{M}},\partial)$ which is an affine map, i.e. $\mc{M}$ can be seen as an open set of $\R^{\dim\mc{M}}$ via $f$. \label{enum:flat3}
\end{enumerate}
\end{definition}

In the following, we call 1-flat (resp. 2-flat) a connection that is flat according to the sense 1 (resp. 2) of the previous definition. A 2-flat connection is clearly 1-flat. Conversely, a 1-flat connection is locally 2-flat, i.e. each point has a neighborhood $\mc{U}$ such that $\nabla$ is 2-flat on $\mc{U}$. A 1-flat connection is a priori not globally 2-flat because the manifold need not be an open set of $\R^n$ (e.g. the circle $\mathbb{S}^1$). The obstruction is topological.

In this work, the flat metrics we introduce on $\SPD(n)$ (which is an open set of the vector space of symmetric matrices) are actually 2-flat.

\subsubsection{Dual connections with respect to a metric}

\begin{definition}[Dual connections]\cite{Amari00}
Let $(\mc{M},g)$ be a Riemannian manifold, $\nabla^g$ the Levi-Civita connection of $g$ and $\nabla,\nabla^*$ be affine connections on $\mc{M}$. We say that $\nabla^*$ is the \textit{dual connection} of $\nabla$ with respect to $g$ if one of the following equivalent requirements is satisfied:
\begin{enumerate}
    \item $\partial_kg_{ij}=g_{lj}\Gamma_{ki}^l+g_{il}(\Gamma^*)^l_{kj}$ for all $i,j,k\in\{1,...,\dim\mc{M}\}$ in any chart, where $\Gamma_{ij}^k$ and $(\Gamma^*)_{ij}^k$ are the Christoffel symbols of $\nabla$ and $\nabla^*$,
    \item $Z\,g(X,Y)=g(\nabla_ZX,Y)+g(X,\nabla^*_ZY)$ for all vector fields $X,Y,Z$ on $\mc{M}$,
    \item $g(X,Y)=g(\Pi X,\Pi^* Y)$ for all vector fields $X,Y$ on $\mc{M}$.
\end{enumerate}
Hence given a metric $g$, $\nabla$ uniquely determines $\nabla^*$ and $(\nabla^*)^*=\nabla$ so that $\frac{\nabla+\nabla^*}{2}$ is a metric connection. We call $(g,\nabla,\nabla^*)$ a \textit{dualistic structure}.\\
Note that if $\nabla$ and $\nabla^*$ are torsion-free, then $\frac{\nabla+\nabla^*}{2}=\nabla^g$.
\end{definition}

\begin{definition}[Dually-flat manifold]\cite{Amari00}
We say that $(\mc{M},g,\nabla,\nabla^*)$ is a dually-flat manifold (or a Hessian manifold) when $\nabla$ and $\nabla^*$ are dual with respect to the metric $g$ and when $\nabla$ and $\nabla^*$ are flat (in the sense 2 of Definition \ref{def:flat_affine_connection}).
\end{definition}

\subsubsection{Divergence}

\begin{definition}[Divergence]\cite{Amari00}
A \textit{divergence} is a distance-like smooth map $D:\mc{M}\times\mc{M}\lto\R_+$ such that:
\begin{enumerate}
    \item (separation) $D(x,y)=0$ if and only if $x=y$,
    \item (non-degenerate) the symmetric positive semi-definite bilinear form $g^D:z\in\mc{M}\lmto -\partial_x|_{x=z}\partial_y|_{y=z}D$ is positive definite. It is called the \textit{induced Riemannian metric}. We denote $\flat:T\M\lto T^*\M$ and $\#=\flat^{-1}:T^*\M\lto T\M$ the musical isomorphisms associated to the metric $g^D$, defined by $\flat(X)(Y)=g(X,Y)$.
\end{enumerate}
We can also define the \textit{dual divergence} $D^*:(x,y)\lmto D(y,x)$ and the \textit{induced connection} by $\nabla^D_XY:z\in\mc{M}\lmto\sharp(Z\lmto\partial^2_x|_{x=z}\partial_y|_{y=z}D(X,Y,Z))$.
\end{definition}

\begin{lemma}[Dual connections induced by a divergence]\cite{Amari00}
Let $D$ be a divergence on $\mc{M}$. Then the connections $\nabla:=\nabla^D$ and $\nabla^*:=\nabla^{D^*}$ are dual with respect to the induced metric $g^D$: a divergence induces a dualistic structure.
\end{lemma}

In general, there is not a canonical way to define a divergence from a dualistic structure, except if it is dually-flat.

\subsubsection{Canonical divergence of a dually-flat manifold}\label{subsubsec:canonical_divergence}

\begin{definition}[Canonical divergence]\label{def:canonical_divergence}\cite{Amari00}
Let $(\mc{M},g,\nabla,\nabla^*)$ be a dually-flat manifold where $\mc{M}$ is simply connected. Let $u,v:\M\lto\R^n$ be two smooth coordinate systems such that $u$ is $\nabla$-affine, $v$ is $\nabla^*$-affine and $g(\frac{\partial}{\partial u^i},\frac{\partial}{\partial v^j})=\delta_{ij}$. The canonical divergence $D$ is defined by $D(x,y)=\psi(x)+\varphi(y)-\dotprod{u(x)}{v(y)}$ for all $x,y\in\mc{M}$ where $\dotprod{\cdot}{\cdot}$ is the canonical inner product on $\R^n$ and $\psi,\varphi:\mc{M}\lto\R$ are smooth maps called potentials defined as follows:
\begin{enumerate}
    \item $d\psi=\sum_i{v^idu^i}$ for all $i\in\{1,...,n\}$ or equivalently without coordinates $d_x\psi(X)=\dotprod{v(x)}{d_xu(X)}$ for all $x\in\mc{M}$ and $X\in T_x\mc{M}$,
    \item $d\varphi=\sum_i{u^idv^i}$ for all $i\in\{1,...,n\}$ or $d_x\varphi(X)=\dotprod{u(x)}{d_xv(X)}$,
    \item $\psi(x)+\varphi(x)=\dotprod{u(x)}{v(x)}$ for all $x\in\mc{M}$.
\end{enumerate}
The equation $d\psi=\sum_i{v^idu^i}$ has a solution by Poincaré's lemma because $\mc{M}$ is simply connected and the differential form $\omega=\sum_i{v^idu^i}$ is closed. Indeed,  $g(\frac{\partial}{\partial u^i},\frac{\partial}{\partial u^j})=\frac{\partial v^k}{\partial u^j}g(\frac{\partial}{\partial u^i},\frac{\partial}{\partial v^k})=\frac{\partial v^i}{\partial u^j}$ and by symmetry of $g$, $g(\frac{\partial}{\partial u^i},\frac{\partial}{\partial u^j})=\frac{\partial v^j}{\partial u^i}$ so $\frac{\partial v^i}{\partial u^j}=\frac{\partial v^j}{\partial u^i}$. So $\psi$ is well defined up to an additive constant and $\varphi$ as well. Finally, $d_x(\psi+\varphi)(X)=\dotprod{v(x)}{d_xu(X)}+\dotprod{u(x)}{d_xv(X)}=d_x(\dotprod{u}{v})(X)$ so there exists a constant $c\in\R$ such that $\psi(x)+\varphi(x)=\dotprod{u(x)}{v(x)}+c$ for all $x\in\mc{M}$. We can impose $c=0$ by choosing the constant in $\varphi$ appropriately.
\end{definition}

\subsection{Principle of balanced bilinear forms}\label{subsec:principle_of_balanced_bilinear_forms}

The principle of balanced bilinear forms \cite{Thanwerdas19-balanced} provides a bilinear form by combining the parallel transports of two flat metrics via the Frobenius inner product. We can give a more general definition of a balanced bilinear form by choosing any inner product on symmetric matrices, although we focus on the Frobenius inner product afterwards.

\begin{principle}[Principle of balanced bilinear forms]
We fix $\dotprod{\cdot}{\cdot}$ an inner product on $\Sym(n)$.
Let $g^+,g^-$ be two flat Riemannian metrics on $\SPD(n)$. We denote $\nabla^+,\nabla^-$ their Levi-Civita connections and $\Pi^+,\Pi^-$ their associated parallel transport maps that do not depend on the curve since the metrics are flat. Then the balanced bilinear form associated to $g^+$ and $g^-$ is defined by:
\begin{equation}
    g^0_\Sigma(X,Y)=\dotprod{\Pi^+_{\Sigma\to I_n}X}{\Pi^-_{\Sigma\to I_n}Y}.
\end{equation}
\end{principle}

\begin{theorem}[Relation between balanced metric and dually-flat manifold]\label{thm:relation_between_balanced_metric_and_dually-flat_manifold} \cite{Thanwerdas19-balanced}
Let $g^+,g^-$ be two flat Riemannian metrics on $\SPD(n)$. We denote $\nabla^+,\nabla^-$ their Levi-Civita connections. If the balanced bilinear form $g^0$ is a metric, then $(\SPD(n),g^0,\nabla^+,\nabla^-)$ is a dually-flat manifold, which automatically comes with a canonical divergence $D$ according to the previous section.
\end{theorem}

It would be nice to have a sufficient condition under which a balanced bilinear form is a metric. In \cite{Thanwerdas19-balanced}, we proved that, with the Frobenius inner product, if $g^+$ and $g^-$ are power-Euclidean metrics with powers $\alpha$ and $\beta$, then $g^0$ is a metric. In the following theorem, we give a weaker sufficient condition which allows to define the new family of Mixed-Euclidean metrics. The proof is in Appendix \ref{app}.

\begin{theorem}[Sufficient condition for a balanced bilinear form to be a metric]\label{thm:suff_cond_balanced_metric_frob}
Let $\dotprod{\cdot}{\cdot}=\Frob$ be the Frobenius inner product. Let $g^+,g^-$ be deformed-Euclidean metrics respectively associated to univariate diffeomorphisms $u$ and $v$. Then the balanced bilinear form $g^0$ is a metric.
\end{theorem}

\section{The new family of mixed-Euclidean metrics}\label{subsec:mixed-euclidean}

\subsection{Definition}

\begin{definition}[Mixed-Euclidean metric $\ME(u,v)$]
The $(u,v)$-Mixed-Euclidean metric is the balanced metric $g^0$ defined in Theorem \ref{thm:suff_cond_balanced_metric_frob}. It is given by:
\begin{equation}
    g^{\ME(u,v)}_\Sigma(X,X)=\frac{1}{u'(1)v'(1)}\sum_{i,j}{u^{[1]}(d_i,d_j)v^{[1]}(d_i,d_j)X_{ij}'^2},
\end{equation}
where $X=PX'P^\top$, $\Sigma=PDP^\top$ with $P\in\Orth(n)$, $D=\diag(d_1,...,d_n)$.
\end{definition}

\begin{remark}
We notice that if we denote $\phi_u=\frac{u'(1)}{u^{[1]}}$ and $\phi_v=\frac{v'(1)}{v^{[1]}}$ the kernel maps associated to the $u,v$-deformed Euclidean metrics, the balanced metric is a kernel metric characterized by $\phi_{u,v}=\sqrt{\phi_u\phi_v}$. Hence, the principle of balanced bilinear forms seems to appear as a principle of mean of metrics.
\end{remark}

The family of Mixed-Euclidean metrics contains the family of Mixed-Power-Euclidean metrics \cite{Thanwerdas19-balanced} for $u=F_\alpha$ and $v=F_\beta$ where $F_\alpha=\pow_\alpha$ if $\alpha\ne 0$ and $F_0=\log$.
\begin{align}
    &(\mathrm{Log}\text{-}\mathrm{Euclidean}) &g^{\MPE(0,0)}_\Sigma(X,X)&=\tr({d_\Sigma\log(X)}^2),\\
    &(\mathrm{Power}\text{-}\mathrm{Euclidean}) &g^{\MPE(\alpha,\alpha)}_\Sigma(X,X)&=\frac{1}{\alpha^2}\tr(d_\Sigma\pow_\alpha(X)^2)
    \label{mpe(p,p)},\\
    &(\mathrm{Power}\text{-}\mathrm{affine}) &g^{\MPE(\alpha,-\alpha)}_\Sigma(X,X)&=\frac{1}{\alpha^2}\tr((\Sigma^{-\alpha}d_\Sigma\pow_\alpha(X))^2)
    \label{mpe(p,-p)},\\
    &(``\mathrm{Power}\text{-}\mathrm{BKM}") &g^{\MPE(\alpha,0)}_\Sigma(X,X)&=\frac{1}{\alpha}\tr(d_\Sigma\pow_\alpha(X)d_\Sigma\log(X)) \label{mpe(p,0)},\\
    &(\mathrm{General~MPE}) &g^{\MPE(\alpha,\beta)}_\Sigma(X,X)&=\frac{1}{\alpha\beta}\tr(d_\Sigma\pow_\alpha(X)d_\Sigma\pow_\beta(X)) \label{mpe(p,q)}.
\end{align}
As mentioned in \cite{Thanwerdas19-balanced}, this family interpolates between the log-Euclidean metric $(0,0)$, the power-Euclidean metrics $(\alpha,\alpha)$, the power-affine metrics $(\alpha,-\alpha)$ (including the affine-invariant metric $(1,-1)$) and the Bogoliubov-Kubo-Mori metric $(1,0)$.

\subsection{Information geometry of Mixed-Euclidean metrics}

As said in Theorem \ref{thm:relation_between_balanced_metric_and_dually-flat_manifold}, balanced metrics come with a canonical divergence. As Mixed-Euclidean metrics are the balanced metrics of two deformed-Euclidean metrics $u^*g^\mathrm{E}$ and $v^*g^\mathrm{E}$, it is straightforward that $u:\SPD(n)\lto\Sym(n)$ and $v:\SPD(n)\lto\Sym(n)$ provide flat coordinate systems for these respective metrics. The canonical divergence of this structure is known as the $(u,v)$-divergence in Information Geometry \cite[Section 4.5.2]{Amari16}. The novelty here is the relation we establish between Mixed-Euclidean metrics and $(u,v)$-divergences. In particular, the Mixed-Power-Euclidean metrics come with the so-called $(\alpha,\beta)$-divergences on SPD matrices \cite{Amari14}. This family contains the well known families of $\alpha$-divergences and $\beta$-divergences \cite[Formulae 69,70]{Amari14}. We state the correspondence between Mixed-Euclidean metrics and $(u,v)$-divergences in the following corollary of Theorem \ref{thm:relation_between_balanced_metric_and_dually-flat_manifold}. We recall the formulae of $(\alpha,\beta)$-divergences with the corresponding potentials and we illustrate the correspondence with two charts.

\begin{corollary}[Mixed-Euclidean metrics and $(u,v)$-divergences]\label{cor:mpe_metrics_and_ab_divergences}
Let $u,v$ be two univariate diffeomorphisms $u,v:\SPD(n)\lto\SPD(n)$. Then the manifold $(\SPD(n),g^{\ME(u,v)},u^*\nabla^{\mathrm{E}},v^*\nabla^{\mathrm{E}})$ is dually-flat and its canonical divergence is the $(u,v)$-divergence of Information Geometry \cite{Amari16}. In particular, the manifold $(\SPD(n),g^{\MPE(p,q)},\nabla^{\mathrm{PE}(p)},\nabla^{\mathrm{PE}(q)})$ is dually-flat and its canonical divergence is the $(\alpha,\beta)$-divergence \cite[Formulae 51,54,56,66]{Amari14}. The $(\alpha,\beta)$-divergences and the corresponding potentials (up to an additive constant, see Section \ref{subsubsec:canonical_divergence}) are:
\small
\begin{align}
    (\alpha=\beta=0)~ &D^{0,0}(\Sigma|\Sigma')=\frac{1}{2}\|\log(\Sigma)-\log(\Sigma')\|_\Frob^2,\\
    (\alpha=\beta\ne 0)~ &D^{\alpha,\alpha}(\Sigma|\Sigma')=\frac{1}{2\alpha^2}\|\Sigma^\alpha-\Sigma'^\alpha\|_\Frob^2,\\
    (\alpha=-\beta\ne 0)~ &D^{\alpha,-\alpha}(\Sigma|\Sigma')=-\frac{1}{\alpha^2}\tr\left[(I_n+\alpha\log\Sigma)-\alpha\log\Sigma'-\Sigma^\alpha\Sigma'^{-\alpha}\right],\\
    (\alpha\ne\beta=0)~ &D^{\alpha,0}(\Sigma|\Sigma')=\frac{1}{\alpha}\tr\left[\left(\Sigma^\alpha\log\Sigma-\frac{1}{\alpha}\Sigma^\alpha\right)+\frac{1}{\alpha}\Sigma'^\alpha-\Sigma^\alpha\log\Sigma'\right],\\
    (\alpha,\beta,\alpha\pm\beta\ne 0)~ &D^{\alpha,\beta}(\Sigma|\Sigma')=\frac{1}{\alpha\beta}\tr\left[\frac{\alpha}{\alpha+\beta}\Sigma^{\alpha+\beta}+\frac{\beta}{\alpha+\beta}\Sigma'^{\alpha+\beta}-\Sigma^\alpha\Sigma'^\beta\right],\\
    &\nonumber\\
    (\alpha=\beta=0)~ &\psi^{0,0}(\Sigma)=\frac{1}{2}\tr(\log(\Sigma)^2),\\
    (\alpha=\beta\ne 0)~ &\psi^{\alpha,\alpha}(\Sigma)=\frac{1}{2\alpha^2}\tr(\Sigma^{2\alpha}),\\
    (\alpha=-\beta\ne 0)~ &\psi^{\alpha,-\alpha}(\Sigma)=-\frac{1}{\alpha}\tr(\log\Sigma)=-\frac{1}{\alpha}\log(\det\Sigma)\\
    (\alpha\ne\beta=0)~ &\psi^{\alpha,0}(\Sigma)=\frac{1}{\alpha}\tr(\Sigma^\alpha\log\Sigma-\frac{1}{\alpha}\Sigma^\alpha),\\
    (\alpha,\beta,\alpha\pm\beta\ne 0)~ &\psi^{\alpha,\beta}(\Sigma)=\frac{1}{\beta(\alpha+\beta)}\tr(\Sigma^{\alpha+\beta}).
\end{align}
\normalsize
The $(\alpha,\beta)$-divergences on SPD matrices can also be obtained by extending the $(\alpha,\beta)$-divergences \textbf{on positive discrete measures} \cite{Cichocki11}. Indeed, a positive discrete measure is a vector of positive numbers so the diagonal matrix of eigenvalues of an SPD matrix can be considered as a positive discrete measure. Then the $(\alpha,\beta)$-potential on positive diagonal matrices is extended by $\Orth(n)$-invariance, which defines the $(\alpha,\beta)$-potential and the $(\alpha,\beta)$-divergence on SPD matrices \cite{Amari14}. Conversely, the $(\alpha,\beta)$-divergences on SPD matrices define divergences on positive discrete measures when restricted to positive diagonal matrices. So there is a one-to-one correspondence between $(\alpha,\beta)$-divergences on SPD matrices \cite{Amari14} (or Mixed-Power-Euclidean metrics) and $(\alpha,\beta)$-divergences on positive discrete measures \cite{Cichocki11}. This correspondence is given on Figure \ref{fig:mpe_metrics_and_ab_divergences}. The graph on the right is essentially borrowed from \cite{Cichocki11} with complements from \cite{Cichocki10}.
\begin{figure}[htbp]
    \centering
    \includegraphics[scale=0.38]{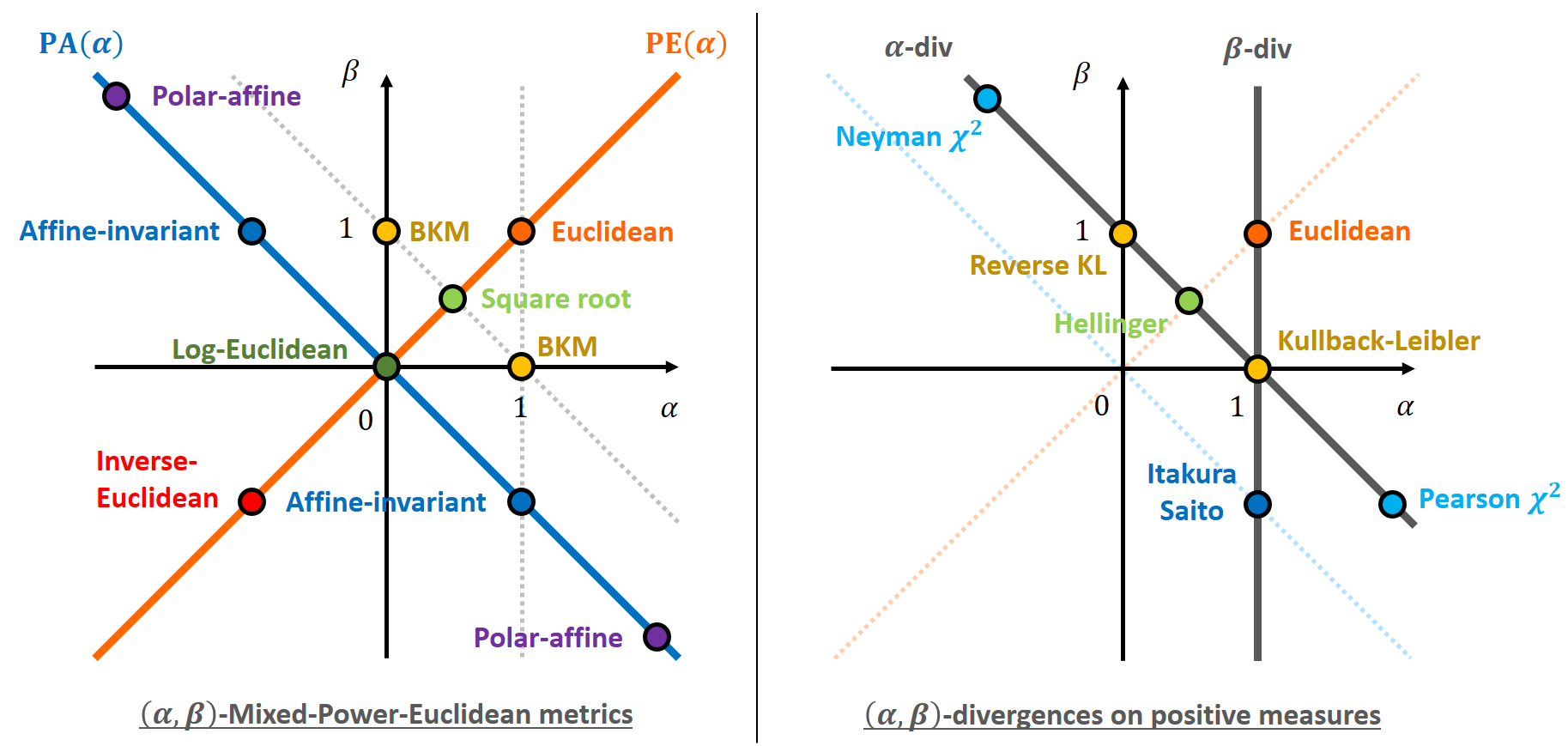}
    \caption{Correspondence between MPE metrics on SPD matrices and $(\alpha,\beta)$-divergences on positive discrete measures.}
    \label{fig:mpe_metrics_and_ab_divergences}
\end{figure}
\end{corollary}

\begin{remark}
The affine-invariant/Fisher-Rao metric is associated to the Kullback-Leibler divergence \textbf{of centered multivariate Gaussian densities}, which differs from the $\Orth(n)$-invariant extension of the Kullback-Leibler divergence \textbf{of positive discrete measures} represented on Figure \ref{fig:mpe_metrics_and_ab_divergences}.
\end{remark}

The $(u,v)$-divergences can be expressed via an integral formula \cite[Formula (4.170)]{Amari16} following Definition \ref{def:canonical_divergence}. The formulae of the previous corollary can thus be computed either from that formula or directly.

\subsection{Riemannian geometry of Mixed-Euclidean metrics}

Another immediate consequence of the relation between balanced metrics and dually-flat manifolds is that the Levi-Civita connection of the Mixed-Euclidean metric $\MPE(u,v)$ is simply the arithmetic mean of the Levi-Civita connections of the deformed-Euclidean metrics $u^*g^{\mathrm{E}}$ and $v^*g^{\mathrm{E}}$.

\begin{corollary}[Levi-Civita connection of Mixed-Euclidean metrics]
\begin{equation}
    \nabla^{\ME(u,v)}_{X_\Sigma}Y=\partial_{X_\Sigma}Y+\frac{1}{2}((d_\Sigma u)^{-1}(H_\Sigma u(X,Y))+(d_\Sigma v)^{-1}(H_\Sigma v(X,Y)))
\end{equation}
\end{corollary}

It is even possible to compute the curvature following the same ideas as for the BKM metric in \cite{Michor00}. The proof is given in Appendix \ref{app}.

\begin{theorem}[Curvature of Mixed-Euclidean metrics]\label{thm:curvature_mixed-euclidean}
Let $u,v:\SPD(n)\lto\SPD(n)$ be two univariate diffeomorphisms. We define the univariate diffeomorphism $w=v\circ u^{-1}$ so that $u:(\SPD(n),g^{\ME(u,v)})\lto(\SPD(n),\frac{w'(1)}{u'(1)v'(1)}g^{\ME(\Id,w)})$ is an isometry. For $\Sigma=PDP^\top\in\SPD(n)$, we denote $X=PX'P^\top\in T_\Sigma\SPD(n)$ and analogously for $Y,Z,T\in T_\Sigma\SPD(n)$, we denote $u_{ij}=u^{[1]}(d_i,d_j)$, $u_{ijk}=u^{[2]}(d_i,d_j,d_k)$ and analogously for $v,w$. We denote $m_{ij}=w^{[1]}(u(d_i),u(d_j))=\frac{v_{ij}}{u_{ij}}$ and $m_{ijk}=w^{[2]}(u(d_i),u(d_j),u(d_k))$. Then the curvature of the mixed-Euclidean metric $g^{\ME(u,v)}$ is:
\begin{align}
    R^{\ME(u,v)}_\Sigma(X,Y,Z,T)=\frac{1}{u'(1)v'(1)}\sum_{i,j,k,l}\rho_{ijkl}&(X'_{ij}Y'_{jk}Z'_{kl}T'_{li}-Y'_{ij}X'_{jk}Z'_{kl}T'_{li}\\
    &+X'_{ij}Z'_{jk}Y'_{kl}T'_{li}-Y'_{ij}Z'_{jk}X'_{kl}T'_{li}),\nonumber
\end{align}
where $\rho_{ijkl}=\frac{m_{ijl}m_{jlk}}{2m_{jl}}u_{ij}u_{jk}u_{kl}u_{li}=\frac{1}{2u_{jl}v_{jl}}(u_{ij}v_{ijl}-v_{ij}u_{ijl})(u_{jk}v_{jkl}-v_{jk}u_{jkl})$ is symmetric in $i\leftrightarrow k$, in $j\leftrightarrow l$ and in $u\leftrightarrow v$. In particular, at $\Sigma=I_n$, the curvature is:
\begin{equation}
    R_{I_n}^{\ME(u,v)}(X,Y,Z,T)=\frac{1}{4}\left[\left(\ln\left|\frac{v'}{u'}\right|\right)'(1)\right]^2 R_{I_n}^{\mathrm{A}}(X,Y,Z,T),
\end{equation}
where $\mathrm{A}$ stands for the affine-invariant metric (Formula \ref{eq:affine-invariant_metric}). Therefore, the sectional curvature of the mixed-Euclidean metric at $I_n$ takes non-positive values. In particular, for mixed-power-Euclidean metrics $\MPE(\alpha,\beta)$ with $\alpha^2\ne\beta^2$ (thus excluding log-Euclidean, power-Euclidean and power-affine metrics), since $\kappa_{\lambda \Sigma}^{\MPE(\alpha,\beta)}(X,Y)=\lambda^{-(\alpha+\beta)}\times\kappa_\Sigma^{\MPE(\alpha,\beta)}(X,Y)$ for all $\lambda>0$, the lower bound of the sectional curvature is $-\infty$.
\end{theorem}

It seems difficult to determine theoretically whether the sectional curvature of mixed-Euclidean metrics (again, excluding $\MPE(\alpha,\beta)$ with $\alpha^2=\beta^2$) can take positive values. On Figure \ref{fig:bounds_curvature_mpe}, we show numerical results which make us think that this is the case. Indeed, we observe numerically that for all $\alpha,\beta\in\{0.05\,k|\,k\in\{-40,...,40\}\}$ such that $\alpha^2\ne\beta^2$, we have $\kappa^{\MPE(\alpha,\beta)}_{\mathrm{min}}<0$ and $\kappa^{\MPE(\alpha,\beta)}_{\mathrm{max}}>0$. These simulations also tend to show that, at a given point $\Sigma$, the negative values taken by the sectional curvature are much larger in absolute value than the positive values taken by the sectional curvature.

From Figure \ref{fig:bounds_curvature_mpe}, it appears that power-Euclidean metrics (flat), power-affine metrics (Hadamard) and the log-Euclidean metric at the intersection play a special role among the family of Mixed-Power-Euclidean since all others apparently admit positive and negative sectional curvature.

\begin{figure}[h]
    \centering
    \includegraphics[scale=.7]{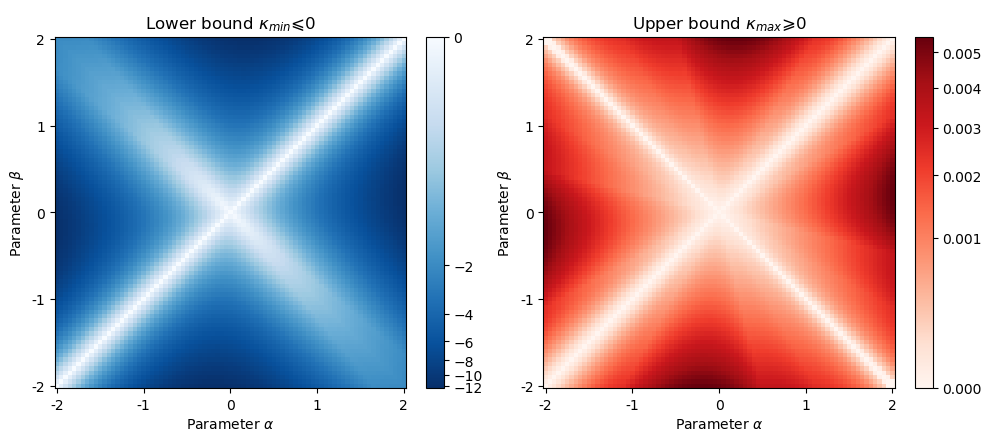}
    \caption{\textbf{Lower and upper bounds of the sectional curvature of the mixed-power-Euclidean metrics}. \textbf{Left}: lower bound. \textbf{Right}: upper bound. The lower bound of the power-affine metrics ($\beta=-\alpha$) is known as $-\frac{\alpha^2}{2}$ \cite{Thanwerdas21-LAA}. The bounds were obtained in dimension $3$ by taking $1000$ random positive diagonal matrices $D$ of determinant $1$ (to avoid scaling effects), $1000$ random pairs of symmetric matrices $(X,Y)$ and computing the sectional curvature $\kappa^{\MPE(\alpha,\beta)}_D(X,Y)=\frac{R_D(X,Y,X,Y)}{g_D(X,X)g_D(Y,Y)-g_D(X,Y)^2}$ for $(\alpha,\beta)\in[-2,2]^2$ with a step $\Delta\alpha=\Delta\beta=0.05$. Diagonal matrices are taken instead of SPD matrices because the MPE metrics are $\Orth(n)$-invariant.}
    \label{fig:bounds_curvature_mpe}
\end{figure}

In addition, for Mixed-Power-Euclidean metrics, we can also compute the geodesics, the logarithm map and the distance between commuting matrices. These formulae are proved in Appendix \ref{app}.

\begin{theorem}[Riemannian operations of MPE metrics]\label{thm:riemannian_operations_of_mpe_metrics}
Let $\alpha,\beta\in\R$ such that $\alpha+\beta\ne 0$, thus excluding log-Euclidean and power-affine metrics. Table \ref{tab:riemannian_operations_mpe} summarizes the formulae of the geodesics, the logarithm map and the distance in the particular case where $\Sigma,\Lambda\in\SPD(n)$ and $V\in T_\Sigma\SPD(n)$ \textbf{commute}. They essentially reduce to the formulae of the $\alpha_0$-power-Euclidean metric with $\alpha_0=\frac{\alpha+\beta}{2}$. These formulae are generally not valid for non-commuting matrices.
\end{theorem}

\vspace{-3mm}
\begin{table}[htbp]
\centering
\begin{tabular}{|c|c|}
    \hline
    Geodesics & $\gamma_{(\Sigma,V)}(t)=(\Sigma^{\alpha_0}+t\,d_\Sigma\pow_{\alpha_0}(V))^{1/\alpha_0}$\\
    \hline
    Logarithm & $\Log_\Sigma(\Lambda)=(d_\Sigma\pow_{\alpha_0})^{-1}(\Lambda^{\alpha_0}-\Sigma^{\alpha_0})$\\
    \hline
    Distance & $d(\Sigma,\Lambda)=\frac{1}{\alpha_0}\|\Lambda^{\alpha_0}-\Sigma^{\alpha_0}\|_\Frob$\\
    \hline
\end{tabular}
\caption{Riemannian operations of Mixed-Power-Euclidean metrics for commuting matrices}
\label{tab:riemannian_operations_mpe}
\end{table}

It would be tempting to generalize the formulae of geodesics, logarithm and distance between commuting matrices to Mixed-Euclidean metrics. However, if we consider two diffeomorphisms $u,v:(0,\infty)\lto(0,\infty)$, the map $f=\sqrt{uv}$ (which generalizes $\pow_{\alpha_0}=\sqrt{\pow_\alpha\pow_\beta}$) is not a diffeomorphism of $(0,\infty)$ in general. For example, take $u(x)=x(x+1)$ and $v(x)=\frac{1}{x^2}$. So the generalization is not straightforward.

\section{Conclusion}

Deforming a Riemannian metric is a general way of defining new metrics and new families of metrics on SPD matrices. In particular, using power diffeomorphisms defines one-parameter families which tend to the log-Euclidean metric when the power tends to 0. The class of kernel metrics is stable by univariate diffeomorphism whereas the class of mean kernel metrics is not. We showed that the alpha-Procrustes (or power-Wasserstein) metrics are mean kernel metrics when the power $p=2\alpha\ls 1$.

We extended the principle of balanced bilinear forms and we gave a new sufficient condition under which the bilinear form is a metric. This allowed to define the new family of Mixed-Euclidean metrics which extends the two-parameter family of Mixed-Power-Euclidean metrics. Since balanced metrics define dually-flat manifolds which are characterized by a canonical divergence, Mixed-Euclidean metrics are in one-to-one correspondence with the $(u,v)$-divergences of information geometry. In particular, Mixed-Power-Euclidean metrics are in bijection with $(\alpha,\beta)$-divergences. Finally, we computed the curvature of all Mixed-Euclidean metrics.

Some questions remain open. What are the conditions on the univariate diffeomorphisms $u,v$ for the $u$-deformed Euclidean, affine, Wasserstein or the $(u,v)$-mixed-Euclidean metric to be a mean kernel metric? Are there more general conditions on two flat metrics for their balanced bilinear form to be a metric? What if we replace the Frobenius inner product by another one? Does the operation $(\phi,\phi')\lmto\sqrt{\phi\phi'}$ on flat kernel metrics generalize the principle of balanced metrics?

More generally, since the two-parameter family of $(\alpha,\beta)$-Mixed-Power-Euclidean metrics interpolate between the Euclidean, the log-Euclidean, the affine-invariant and the Bogoliubov-Kubo-Mori metrics, does there exist a principled family with three or four parameters which additionally includes the Bures-Wasserstein metric? It is not difficult to build parametric families of metrics which interpolate between all of them whereas it is more difficult to find interpolations with an interesting geometry such as the dually-flat geometry.

\section*{Acknowledgements}
This project has received funding from the European Research Council (ERC) under the European Union’s Horizon 2020 research and innovation program (grant G-Statistics agreement No 786854). This work has been supported by the French government, through the UCAJEDI Investments in the Future project managed by the National Research Agency (ANR) with the reference number ANR-15-IDEX-01 and through the 3IA Côte d’Azur Investments in the Future project managed by the National Research Agency (ANR) with the reference number ANR-19-P3IA-0002. The authors warmly thank Frank Nielsen for insightful discussions on $(u,v)$-divergences and Nicolas Guigui and Dimbihery Rabenoro for their careful proofreading of this manuscript.

\begin{appendices}
\section{Proofs}\label{app}

\begin{proof}[Proof of Theorem \ref{thm:power-wasserstein} (Condition for power-Wasserstein to be a mean kernel metric)]

Let us show that if $p\ls 1$, then the function $m(x,y)=\left[(x^p+y^p)\left(p\frac{x-y}{x^p-y^p}\right)^2\right]^{\frac{1}{2-p}}$ is non-decreasing in $x$ (and by symmetry in $y$). If we factorize by $y$ and define a new variable $t=x/y$, we have to study the function $f(t)=F(t)^{\frac{1}{2-p}}$ where $F(t)=(t^p+1)\left(p\frac{t-1}{t^p-1}\right)^2$.

First, let us prove that $F$ is non-decreasing if and only if $p\ls 1$. We denote $f_0(x)=p\frac{x-1}{x^p-1}>0$ and $g_0(x)=x^pf_0(x)^2$ so that $F(x)=f_0(x)^2+g_0(x)$. Note that $f_0$ is non-decreasing if and only if $p\ls 1$. We also introduce $h_0(x)=(x^p+1)(x-1)^2$ so that $F(x)=p^2\frac{h_0(x)}{(x^p-1)^2}$.
\begin{enumerate}
    \item If $p-1>0$, then $F'(0)=-2p^2<0$ so $F'$ cannot be positive around $0$ because it is smooth at $0$. So $F$ is not non-decreasing.
    \item If $p\in]0,1]$, then $F$ is non-decreasing as product of three non-decreasing positive functions.
    \item We assume that $p<0$. Let us prove separately that $F$ is increasing on $(0,1)$ and on $(1,\infty)$. As $F$ is continuous (at $1$), it will prove that $F$ is increasing on $(0,\infty)$.
    \begin{enumerate}
        \item Let us prove that $F$ is increasing on $(0,1)$. We only need to prove that $g_0$ is increasing on $(0,1)$. We successively derive:
        \begin{enumerate}[label=$\bullet$]
            \item $g_0'(x)=\underset{<0}{\underbrace{p^2\frac{x-1}{(x^p-1)^3}x^{p-1}}}\underset{g_1(x)}{\underbrace{((2-p)x^{p+1}+px^p-(p+2)x+p)}}$,
            \item $g_1'(x)=(2-p)(p+1)x^p+p^2x^{p-1}-(p+2)$,
            \item $g_1''(x)=px^{p-2}((2-p)(p+1)x+p(p-1))$.
        \end{enumerate}
        We can notice that $g_1(1)=0$ and $g_1'(1)=0$. So if we prove that $g_1''<0$ on $(0,1)$, then $g_1'$ is decreasing thus positive, so $g_1$ is increasing thus negative, and finally $g_0'$ is positive so $g_0$ increases.
        \begin{enumerate}
            \item If $p+1=0$, then $g_1''$ has the sign of $p^2(p-1)<0$,
            \item if $p+1>0$, then $g_1''$ is positive before $x_0:=\frac{p(p-1)}{(p+1)(p-2)}<0$ and negative after,
            \item if $p+1<0$, then $g_1''$ is negative before $x_0:=\frac{p(p-1)}{(p+1)(p-2)}>1$ and positive after.
        \end{enumerate}
        So we proved that $F$ is increasing on $(0,1)$.
        
        \item Let us prove that $F$ is increasing on $(1,\infty)$. As $x\lmto\frac{1}{(x^p-1)^2}$ is increasing, we only need to prove that $h_0$ is increasing on $(1,\infty)$ so that $F$ is increasing on $(1,\infty)$ as product of two positive increasing functions. We derive successively:
        \begin{enumerate}[label=$\bullet$]
            \item $h_0'(x)=(x-1)h_1(x)$ with $h_1(x)=(p+2)x^p-px^{p-1}+2$,
            \item $h_1'(x)=px^{p-2}((p+2)x-(p-1))$.
        \end{enumerate}
        We need to prove that $h_1>0$ on $(1,\infty)$. As before, we distinguish the cases:
        \begin{enumerate}
            \item If $p+2<0$, then $h_1'$ is negative before $x_0:=\frac{p-1}{p+2}>1$ and positive after. As $h_1(x_0)=2-x_0^{p-1}>1$, we have $h_1>0$ on $(1,\infty)$.
            \item If $p+2\gs 0$, then $h_1'$ is negative on $(1,\infty)$. As $\lim_{x\to\infty}{h_1(x)}=2$, we have $h_1>0$ on $(1,\infty)$.
        \end{enumerate}
        So we proved that $F$ is increasing on $(1,\infty)$ and therefore on $(0,\infty)$.
    \end{enumerate}
\end{enumerate}
Finally, we proved that $F$ is non-decreasing if and only if $p\ls 1$. As $f=\pow_{\frac{1}{2-p}}\circ F$, we can assert that if $p\ls 1$, then $f$ is non-decreasing, as expected.
\end{proof}

\begin{proof}[Proof of Theorem \ref{thm:suff_cond_balanced_metric_frob} (Sufficient condition for a balanced bilinear form to be a metric)]
Let $u,v:\SPD(n)\lto\Sym(n)$ be two univariate diffeomorphisms onto the respective image of $\SPD(n)$ by $u$ and $v$. Let $g^+,g^-$ be the respective deformed-Euclidean metrics by $u$ and $v$. Hence for all $\Sigma\in\SPD(n)$ and $X,Y\in T_\Sigma\SPD(n)$:
\begin{align}
    g^+_\Sigma(X,X)&=\tr(d_\Sigma u(X)^2),\\
    g^-_\Sigma(Y,Y)&=\tr(d_\Sigma v(Y)^2).
\end{align}
Hence the flat parallel transports $\Pi^+$ and $\Pi^-$ do not depend on the curve, they are simply given by the differentials of $u$ and $v$:
\begin{align}
    \Pi^+_{\Sigma\to\Lambda}X=(d_\Lambda u)^{-1}(d_\Sigma u(X)),\\
    \Pi^-_{\Sigma\to\Lambda}Y=(d_\Lambda v)^{-1}(d_\Sigma v(Y)),
\end{align}
where $d_\Sigma u(X)=P\,d_D u(P^\top XP)P^\top$ and $[d_Du(P^\top XP)]_{ij}=u^{[1]}(d_i,d_j)[P^\top XP]_{ij}$ given an eigenvalue decomposition of $\Sigma=PDP^\top$ with $D=\diag(d_1,...,d_n)\in\Diag^+(n)$ and $P\in\Orth(n)$. Note that $d_{I_n}u(X)=u'(1)X$ so $d_{I_n}u=u'(1)\Id$. The same is valid for $v$. Therefore, since the parallel transport is $\Orth(n)$-equivariant, the balanced bilinear form $g^0$ is defined by:
\begin{align}
    g^0_\Sigma(X,Y)&=\tr((\Pi^+_{\Sigma\to I_n}X)(\Pi^-_{\Sigma\to I_n}Y))\\
    &=\tr(P(\Pi^+_{D\to I_n}P^\top XP)P^\top P(\Pi^-_{D\to I_n}P^\top YP)P^\top)\\
    &=\frac{1}{u'(1)v'(1)}\tr(d_Du(P^\top XP)d_Dv(P^\top YP))\\
    &=\frac{1}{u'(1)v'(1)}\sum_{i,j}{u^{[1]}(d_i,d_j)v^{[1]}(d_i,d_j)[P^\top XP]_{ij}[P^\top YP]_{ij}}.
\end{align}
First, $g^0$ is symmetric. Second, since $u:(0,\infty)\lto\R$ is a diffeomorphism, either $u'>0$ or $u'<0$ and by the mean value theorem, the sign of $u^{[1]}$ is the sign of $u'$. Hence $\frac{u^{[1]}(d_i,d_j)}{u'(1)}>0$ and similarly $\frac{v^{[1]}(d_i,d_j)}{v'(1)}>0$ so the coefficients of the quadratic form $g^0_\Sigma(X,X)$ are positive. So the balanced bilinear form $g^0$ is a Riemannian metric.
\end{proof}

\begin{proof}[Proof of Theorem \ref{thm:curvature_mixed-euclidean} (Curvature of Mixed-Euclidean metrics)]

We compute the curvature of the metric $g^{\ME(u,v)}$ for univariate diffeomorphisms $u,v:\SPD(n)\lto\SPD(n)$. Since $d_\Sigma v(X)=d_{u(\Sigma)}w(d_\Sigma u(X))$ with $w=v\circ u^{-1}$, the map $u$ is an isometry between $g^{\ME(u,v)}$ and $c\,g^{\ME(\Id,w)}$ with $c=\frac{w'(1)}{u'(1)v'(1)}$. So it suffices to compute the curvature of $g^{\ME(\Id,w)}$ and to conclude by pullback and scaling.

Let $\Sigma=PDP^\top\in\SPD(n)$. We denote $u_{ij}=u^{[1]}(d_i,d_j)$, $u_{ijk}=u^{[2]}(d_i,d_j,d_k)$ and analogously for $v$ and $w$.

The curvature of $g:=g^{\ME(\Id,w)}$ can be computed the same way as shown in \cite{Michor00} for the metric $\BKM=\MPE(1,0)=\ME(\Id,\log)$. Following \cite{Michor00}, we introduce $G_\Sigma(X)=d_\Sigma w(X)$ and $\Gamma_\Sigma(X,Y)$ such that:
\begin{align}
    g_\Sigma(X,Y) &=\frac{1}{w'(1)}\tr(d_\Sigma w(X)Y)=\frac{1}{w'(1)}\tr(G_\Sigma(X)Y),\\
    \nabla_{X_\Sigma}Y &= d_\Sigma Y(X)+\Gamma_\Sigma(X,Y),
\end{align}
where $\nabla$ is the Levi-Civita connection of $g$. Note that $G_\Sigma:\Sym(n)\lto\Sym(n)$ is a linear isomorphism and $\Gamma_\Sigma$ is symmetric. According to Lemma \ref{lemma_diff_univariate}, $[G_D(X)]_{ij}=w_{ij}X_{ij}$. Then the Riemann curvature tensors are defined by $R(X,Y)Z=\nabla_X\nabla_YZ-\nabla_Y\nabla_XZ-\nabla_{[X,Y]}Z$ and $R(X,Y,Z,T)=-g(R(X,Y)Z,T)$ so we can write $R$ in function of $\Gamma$ and $d\Gamma$ or in function of $G$ and $dG$ \cite{Michor00}:
\begin{align}
    R(X,Y)Z&=d\Gamma(X)(Y,Z)-d\Gamma(Y)(X,Z)+\Gamma(X,\Gamma(Y,Z))-\Gamma(Y,\Gamma(X,Z))\\
    &=-\frac{1}{4}G^{-1}(dG(X)(G^{-1}(dG(Y)(Z))))+\frac{1}{4}G^{-1}(dG(Y)(G^{-1}(dG(X)(Z)))),\\
    R(X,Y,Z,T)&=\frac{1}{4w'(1)}\tr\left[\left(dG(X)(G^{-1}(dG(Y)(Z)))-dG(Y)(G^{-1}(dG(X)(Z)))\right)T\right].
\end{align}

So we only need to express $d_\Sigma G(X)(Y)=H_\Sigma w(X,Y)$. Lemma \ref{lemma_diff_univariate} gives $H_\Sigma w(X,Y)=P\,H_Dw(X',Y')P^\top$ and $[H_D(X',Y')]_{ij}=\sum_kw_{ijk}(X'_{ik}Y'_{jk}+X'_{jk}Y'_{ik})$. Hence:
\small
\begin{align}
    \tr(d_\Sigma G(X)(G_\Sigma^{-1}(d_\Sigma G(Y)(Z)))T)&=\tr(d_DG(X')(G_D^{-1}(d_DG(Y')(Z')))T')\\
    &=\sum_{i,j}[d_DG(X')(G_D^{-1}(d_DG(Y')(Z')))]_{ij}T'_{ij}\\
    &=\sum_{i,j,k}w_{ijk}(X'_{ik}[G_D^{-1}(d_DG(Y')(Z'))]_{jk}+X'_{jk}[G_D^{-1}(d_DG(Y')(Z'))]_{ik})T'_{ij}\\
    &=2\sum_{i,j,k}w_{ijk}X'_{ik}[G_D^{-1}(d_DG(Y')(Z'))]_{jk}T'_{ij}\\
    &=2\sum_{i,j,k}\frac{w_{ijk}}{w_{jk}}X'_{ik}[d_DG(Y')(Z')]_{jk}T'_{ij}\\
    &=2\sum_{i,j,k,l}\frac{w_{ijk}w_{jkl}}{w_{jk}}X'_{ik}(Y'_{jl}Z'_{kl}+Y'_{kl}Z'_{jl})T'_{ij}.
\end{align}
\normalsize
Therefore:
\begin{align}
    R(X,Y,Z,T)&=\frac{1}{w'(1)}\sum_{i,j,k,l}\frac{w_{ijk}w_{jkl}}{2w_{jk}}(X'_{ik}Y'_{kl}Z'_{lj}T'_{ji}+X'_{ik}Z'_{kl}Y'_{lj}T'_{ji}\\
    &\quad\quad\quad\quad\quad\quad\quad\quad\quad\quad -Y'_{ik}X'_{kl}Z'_{lj}T'_{ji}-Y'_{ik}Z'_{kl}X'_{lj}T'_{ji}).\nonumber
\end{align}

To get the curvature of $g^{\ME(u,v)}$, we scale this formula by $c=\frac{w'(1)}{u'(1)v'(1)}$ and we pull it back via the map $u$. The coefficients $w_{ij}=w^{[1]}(d_i,d_j)$ and $w_{ijk}=w^{[2]}(d_i,d_j,d_k)$ are replaced by the coefficients $m_{ij}=w^{[1]}(u(d_i),u(d_j))=v_{ij}/u_{ij}$ and $m_{ijk}=w^{[2]}(u(d_i),u(d_j),u(d_k))$. The vectors $X=[X_{ij}]_{i,j}$ are replaced by the vectors $d_\Sigma u(X)=[u_{ij}X_{ij}]_{i,j}$. Hence the curvature of $\ME(u,v)$ writes (modulo a permutation of indexes $l\to k\to j\to l$):
\begin{align}
    R^{\ME(u,v)}_\Sigma(X,Y,Z,T)&=\frac{1}{u'(1)v'(1)}\sum_{i,j,k,l}\rho_{ijkl}(X'_{ij}Y'_{jk}Z'_{kl}T'_{li}+X'_{ij}Z'_{jk}Y'_{kl}T'_{li} \label{eq:proof_curvature}\\
    &\quad\quad\quad\quad\quad\quad\quad\quad\quad\quad-Y'_{ij}X'_{jk}Z'_{kl}T'_{li}-Y'_{ij}Z'_{jk}X'_{kl}T'_{li}),\nonumber
\end{align}
where $\rho_{ijkl}=\frac{m_{ijl}m_{jlk}}{2m_{jl}}u_{ij}u_{jk}u_{kl}u_{li}$. This expression is symmetric in $i\leftrightarrow k$ and $j\leftrightarrow l$ but it does not look really symmetric in $u\leftrightarrow v$. Let us check that it is though. If $d_j\ne d_l$, we can write:
\begin{align*}
    \rho_{ijkl}&= \frac{1}{2}\frac{u_{jl}}{v_{jl}}\frac{1}{u(d_j)-u(d_l)}(m_{ij}-m_{il})\frac{1}{u(d_j)-u(d_l)}(m_{kj}-m_{kl})u_{ij}u_{jk}u_{kl}u_{li}\\
    &= \frac{1}{2(d_j-d_l)^2}\frac{1}{u_{jl}v_{jl}}(v_{ij}u_{il}-u_{ij}v_{il})(v_{kj}u_{kl}-u_{kj}v_{kl})\\
    &= \frac{(v_{ij}(u_{il}-u_{ij})-u_{ij}(v_{il}-v_{ij}))(v_{kj}(u_{kl}-u_{kj})-u_{kj}(v_{kl}-v_{kj}))}{2(d_j-d_l)^2u_{jl}v_{jl}}\\
    &= \frac{1}{2u_{jl}v_{jl}}(u_{ij}v_{ijl}-v_{ij}u_{ijl})(u_{jk}v_{jkl}-v_{jk}u_{jkl}).
\end{align*}
Since the two expressions of $\rho_{ijkl}$ are continuous in $(d_i,d_j,d_k,d_l)$, they also coincide when $d_j=d_l$. The last expression is clearly symmetric in $u\leftrightarrow v$.

The curvature at $\Sigma=D=I_n$ follows from the following computations:
\begin{align}
    u_{ij}&=u^{[1]}(d_i,d_j)=u'(1),\\
    m_{ij}&=\frac{v_{ij}}{u_{ij}}=\frac{v'(1)}{u'(1)},\\
    m_{ijk}&=w^{[2]}(u(d_i),u(d_j),u(d_k))=\frac{1}{2}w''(u(1)),\\
    w'&=(v'\circ u^{-1})\times(u^{-1})'=\frac{v\circ u^{-1}}{u'\circ u^{-1}},\\
    w'\circ u&=\frac{v'}{u'},\\
    (w''\circ u)\times u'&=\frac{v''u'-v'u''}{{u'}^2},\\
    m_{ijk}&=\left(\frac{v''u'-u'v''}{2{u'}^3}\right)(1),\\
    \rho_{ijkl}&=\frac{m_{ijl}m_{jlk}}{2m_{jl}}u_{ij}u_{jk}u_{kl}u_{li}\\
    &=\frac{u'(1)}{2v'(1)}\left(\frac{v''(1)u'(1)-u'(1)v''(1)}{2{u'(1)}^3}\right)^2{u'(1)}^4\\
    &=\frac{1}{8u'(1)v'(1)}(v''(1)u'(1)-v'(1)u''(1))^2\\
    &=\frac{1}{8u'(1)v'(1)}u'(1)^2v'(1)^2\left(\frac{v''(1)}{v'(1)}-\frac{u''(1)}{u'(1)}\right)^2\\
    &=\frac{1}{8}u'(1)v'(1)\left((\ln|v'|)'-(\ln|u'|)'\right)^2(1)\\
    &=\frac{1}{8}u'(1)v'(1)\left[\left(\ln\left|\frac{v'}{u'}\right|\right)'(1)\right]^2.
\end{align}
Hence, according to Formula (\ref{eq:proof_curvature}), the curvature at $I_n$ writes:
\small
\begin{align}
    R^{\ME(u,v)}_{I_n}(X,Y,Z,T)&=\frac{1}{8}\left[\left(\ln\left|\frac{v'}{u'}\right|\right)'(1)\right]^2\,\tr(XYZT+XZYT-YXZT-YZXT)\\
    &=\frac{1}{8}\left[\left(\ln\left|\frac{v'}{u'}\right|\right)'(1)\right]^2\,\tr(XYZT-YXZT),
\end{align}
\normalsize
because the second and fourth terms cancel. Recognizing the curvature of the affine-invariant $R_{I_n}^{\mathrm{A}}(X,Y,Z,T)=\frac{1}{2}\tr(XYZT-YXZT)$ metric \cite{Skovgaard84,Pennec20,Thanwerdas21-LAA}, we can finally write:
\begin{equation}
    R^{\ME(u,v)}_{I_n}(X,Y,Z,T)=\frac{1}{4}\left[\left(\ln\left|\frac{v'}{u'}\right|\right)'(1)\right]^2R_{I_n}^{\mathrm{A}}(X,Y,Z,T).
\end{equation}
\end{proof}

\begin{proof}[Proof of Theorem \ref{thm:riemannian_operations_of_mpe_metrics} (Riemannian operations of Mixed-Power-Euclidean metrics)]
We compute the geodesics, the logarithm map and the distance between commuting matrices. We show that the geodesics of the Mixed-Power-Euclidean metrics $\MPE(\alpha,\beta)$ with $\alpha+\beta\ne 0$ when the base point $\Sigma\in\SPD(n)$ and the initial tangent vector $X\in T_\Sigma\SPD(n)$ commute is $\gamma(t)=(\Sigma^{\alpha_0}+t\,\alpha_0\,\Sigma^{\alpha_0-1}X)^{1/p_0}$ where $\alpha_0=\frac{\alpha+\beta}{2}\ne 0$. Once this is shown, the formulae of the logarithm and the distance are obvious so we omit the proofs. As the metric is $\Orth(n)$-invariant, we can assume that $\Sigma$ and $X$ are diagonal matrices.

First, we assume that $\alpha,\beta\ne 0$. As $\MPE(\alpha,\beta)$ is a balanced metric, the Levi-Civita connection is $\nabla^{\MPE(\alpha,\beta)}=\frac{1}{2}(\pow_\alpha^*\nabla^{\mathrm{E}}+\pow_\beta^*\nabla^{\mathrm{E}})$ where $\nabla^{\mathrm{E}}$ is the Euclidean connection on symmetric matrices. Since for any curve $\gamma$ on $\SPD(n)$, we have:
\begin{align}
    (\pow_\alpha^*\nabla^{\mathrm{E}})_{\gamma'(t)}\gamma'&=(d_{\gamma(t)}\pow_\alpha)^{-1}(\nabla^{\mathrm{E}}_{(d_{\gamma(t)}\pow_\alpha)(\gamma'(t))}d\pow_\alpha(\gamma'))\\
    &=(d_{\gamma(t)}\pow_\alpha)^{-1}(\nabla^{\mathrm{E}}_{(\gamma^\alpha)'(t)}(\gamma^\alpha)')\\
    &=(d_{\gamma(t)}\pow_\alpha)^{-1}((\gamma^\alpha)''(t)),
\end{align}
the geodesic equation $\nabla^{\MPE(\alpha,\beta)}_{\gamma'}\gamma'=0$ rewrites:
\begin{equation}
    (d_{\gamma(t)}\pow_\alpha)^{-1}((\gamma^\alpha)''(t))+(d_{\gamma(t)}\pow_\beta)^{-1}((\gamma^\beta)''(t))=0.
\end{equation}
We compute:
\begin{align}
    \gamma(t)^\alpha&=\Sigma^\alpha(I_n+t\,\alpha_0\,\Sigma^{-1}X)^{\frac{\alpha}{\alpha_0}},\\
    (\gamma^\alpha)'(t)&=\alpha\,\Sigma^{\alpha-1}X(I_n+t\,\alpha_0\,\Sigma^{-1}X)^{\frac{\alpha}{\alpha_0}-1},\\
    (\gamma^\alpha)''(t)&=\alpha(\alpha-\alpha_0)\,\Sigma^{\alpha-2}X^2(I_n+t\,\alpha_0\,\Sigma^{-1}X)^{\frac{\alpha}{\alpha_0}-2},\\
    (d_{\gamma(t)}\pow_\alpha)^{-1}((\gamma^\alpha)''(t))&=\frac{1}{\alpha}\gamma(t)^{1-\alpha}(\gamma^\alpha)''(t)\\
    &=\frac{\alpha-\beta}{2}\Sigma^{-1}X^2(I_n+t\,\alpha_0\,\Sigma^{-1}X)^{\frac{1}{\alpha_0}-2}.\label{eq:alpha_second_derivative}
\end{align}
As this expression is skew-symmetric in $(\alpha,\beta)$, the curve $\gamma$ satisfies the geodesic equation.

Second, we assume that $\alpha\ne 0$ and $\beta=0$. Similarly, the Levi-Civita connection is $\nabla^{\MPE(\alpha,\beta)}=\frac{1}{2}(\pow_\alpha^*\nabla^{\mathrm{E}}+\log^*\nabla^{\mathrm{E}})$. Hence the geodesic equation is analogously $(d_{\gamma(t)}\pow_\alpha)^{-1}((\gamma^\alpha)''(t))+(d_{\gamma(t)}\log)^{-1}((\log\gamma)''(t)=0$. Thus:
\begin{align}
    \log\gamma(t)&=\log\Sigma+\frac{1}{\alpha_0}\log(I_n+t\,\alpha_0\,\Sigma^{-1}X),\\
    (\log\gamma)'(t)&=\Sigma^{-1}X(I_n+t\,\alpha_0\,\Sigma^{-1}X)^{-1},\\
    (\log\gamma)''(t)&=-\alpha_0\,\Sigma^{-2}X^2(I_n+t\,\alpha_0\,\Sigma^{-1}X)^{-2},\\
    (d_{\gamma(t)}\log)^{-1}((\log\gamma)''(t))&=\gamma(t)(\log\gamma)''(t)\\
    &=-\frac{\alpha}{2}\Sigma^{-1}X^2(I_n+t\,\alpha_0\,\Sigma^{-1}X)^{\frac{1}{\alpha_0}-2}.
\end{align}
This expression cancels Equation (\ref{eq:alpha_second_derivative}) with $\beta=0$ so the curve $\gamma$ is the geodesic.
\end{proof}

\end{appendices}

\bibliographystyle{elsarticle-num-names}
\bibliography{biblio}

\begin{thebibliography}{35}
\expandafter\ifx\csname natexlab\endcsname\relax\def\natexlab#1{#1}\fi
\providecommand{\url}[1]{\texttt{#1}}
\providecommand{\href}[2]{#2}
\providecommand{\path}[1]{#1}
\providecommand{\DOIprefix}{doi:}
\providecommand{\ArXivprefix}{arXiv:}
\providecommand{\URLprefix}{URL: }
\providecommand{\Pubmedprefix}{pmid:}
\providecommand{\doi}[1]{\href{http://dx.doi.org/#1}{\path{#1}}}
\providecommand{\Pubmed}[1]{\href{pmid:#1}{\path{#1}}}
\providecommand{\bibinfo}[2]{#2}
\ifx\xfnm\relax \def\xfnm[#1]{\unskip,\space#1}\fi
\bibitem[{Skovgaard(1984)}]{Skovgaard84}
\bibinfo{author}{L.~T. Skovgaard},
\newblock \bibinfo{title}{{A Riemannian Geometry of the Multivariate Normal
  Model}},
\newblock \bibinfo{journal}{Scandinavian Journal of Statistics}
  \bibinfo{volume}{11} (\bibinfo{year}{1984}) \bibinfo{pages}{211--223}.
\bibitem[{Amari and Nagaoka(2000)}]{Amari00}
\bibinfo{author}{S.~Amari}, \bibinfo{author}{H.~Nagaoka},
  \bibinfo{title}{Methods of Information Geometry}, volume
  \bibinfo{volume}{191}, \bibinfo{publisher}{Oxford University Press},
  \bibinfo{year}{2000}.
\bibitem[{Moakher(2005)}]{Moakher05}
\bibinfo{author}{M.~Moakher},
\newblock \bibinfo{title}{A {Differential} {Geometric} {Approach} to the
  {Geometric} {Mean} of {Symmetric} {Positive}-{Definite} {Matrices}},
\newblock \bibinfo{journal}{SIAM Journal on Matrix Analysis and Applications}
  \bibinfo{volume}{26} (\bibinfo{year}{2005}) \bibinfo{pages}{735--747}.
\bibitem[{Pennec et~al.(2006)Pennec, Fillard, and Ayache}]{Pennec06}
\bibinfo{author}{X.~Pennec}, \bibinfo{author}{P.~Fillard},
  \bibinfo{author}{N.~Ayache},
\newblock \bibinfo{title}{A {Riemannian} {Framework} for {Tensor} {Computing}},
\newblock \bibinfo{journal}{International Journal of Computer Vision}
  \bibinfo{volume}{66} (\bibinfo{year}{2006}) \bibinfo{pages}{41--66}.
\bibitem[{Lenglet et~al.(2006)Lenglet, Rousson, Deriche, and
  Faugeras}]{Lenglet06-JMIV}
\bibinfo{author}{C.~Lenglet}, \bibinfo{author}{M.~Rousson},
  \bibinfo{author}{R.~Deriche}, \bibinfo{author}{O.~Faugeras},
\newblock \bibinfo{title}{Statistics on the {Manifold} of {Multivariate}
  {Normal} {Distributions}: {Theory} and {Application} to {Diffusion} {Tensor}
  {MRI} {Processing}},
\newblock \bibinfo{journal}{Journal of Mathematical Imaging and Vision}
  \bibinfo{volume}{25} (\bibinfo{year}{2006}) \bibinfo{pages}{423--444}.
\bibitem[{Fletcher and Joshi(2007)}]{Fletcher07}
\bibinfo{author}{P.~T. Fletcher}, \bibinfo{author}{S.~Joshi},
\newblock \bibinfo{title}{{Riemannian Geometry for the Statistical Analysis of
  Diffusion Tensor Data}},
\newblock \bibinfo{journal}{Signal Processing} \bibinfo{volume}{87}
  (\bibinfo{year}{2007}) \bibinfo{pages}{250–262}.
\bibitem[{Arsigny et~al.(2006)Arsigny, Fillard, Pennec, and Ayache}]{Arsigny06}
\bibinfo{author}{V.~Arsigny}, \bibinfo{author}{P.~Fillard},
  \bibinfo{author}{X.~Pennec}, \bibinfo{author}{N.~Ayache},
\newblock \bibinfo{title}{{Log-Euclidean metrics for fast and simple calculus
  on diffusion tensors.}},
\newblock \bibinfo{journal}{{Magnetic Resonance in Medicine}}
  \bibinfo{volume}{56} (\bibinfo{year}{2006}) \bibinfo{pages}{411--21}.
\bibitem[{Dowson and Landau(1982)}]{Dowson82}
\bibinfo{author}{D.~Dowson}, \bibinfo{author}{B.~Landau},
\newblock \bibinfo{title}{The fréchet distance between multivariate normal
  distributions},
\newblock \bibinfo{journal}{Journal of Multivariate Analysis}
  \bibinfo{volume}{12} (\bibinfo{year}{1982}) \bibinfo{pages}{450--455}.
\bibitem[{Olkin and Pukelsheim(1982)}]{Olkin82}
\bibinfo{author}{I.~Olkin}, \bibinfo{author}{F.~Pukelsheim},
\newblock \bibinfo{title}{The distance between two random vectors with given
  dispersion matrices},
\newblock \bibinfo{journal}{Linear Algebra and its Applications}
  \bibinfo{volume}{48} (\bibinfo{year}{1982}) \bibinfo{pages}{257--263}.
\bibitem[{Dryden et~al.(2009)Dryden, Koloydenko, and Zhou}]{Dryden09}
\bibinfo{author}{I.~L. Dryden}, \bibinfo{author}{A.~Koloydenko},
  \bibinfo{author}{D.~Zhou},
\newblock \bibinfo{title}{{Non-Euclidean statistics for covariance matrices,
  with applications to diffusion tensor imaging}},
\newblock \bibinfo{journal}{The Annals of Applied Statistics}
  \bibinfo{volume}{3} (\bibinfo{year}{2009}) \bibinfo{pages}{1102--1123}.
\bibitem[{Takatsu(2010)}]{Takatsu10}
\bibinfo{author}{A.~Takatsu},
\newblock \bibinfo{title}{{On Wasserstein geometry of Gaussian measures}},
\newblock in: \bibinfo{editor}{M.~Kotani}, \bibinfo{editor}{M.~Hino},
  \bibinfo{editor}{T.~Kumagai} (Eds.), \bibinfo{booktitle}{Probabilistic
  Approach to Geometry}, volume~\bibinfo{volume}{57} of
  \textit{\bibinfo{series}{Advanced Studies in Pure Mathematics}},
  \bibinfo{publisher}{Mathematical Society of Japan}, \bibinfo{address}{Kyoto
  University, Japan}, \bibinfo{year}{2010}, pp. \bibinfo{pages}{463--472}.
\bibitem[{Takatsu(2011)}]{Takatsu11}
\bibinfo{author}{A.~Takatsu},
\newblock \bibinfo{title}{{Wasserstein geometry of Gaussian measures}},
\newblock \bibinfo{journal}{Osaka Journal of Mathematics} \bibinfo{volume}{48}
  (\bibinfo{year}{2011}) \bibinfo{pages}{1005--1026}.
\bibitem[{Bhatia et~al.(2019)Bhatia, Jain, and Lim}]{Bhatia19}
\bibinfo{author}{R.~Bhatia}, \bibinfo{author}{T.~Jain},
  \bibinfo{author}{Y.~Lim},
\newblock \bibinfo{title}{On the {Bures}–{Wasserstein} distance between
  positive definite matrices},
\newblock \bibinfo{journal}{Expositiones Mathematicae} \bibinfo{volume}{37}
  (\bibinfo{year}{2019}) \bibinfo{pages}{165--191}.
\bibitem[{Petz and Toth(1993)}]{Petz93}
\bibinfo{author}{D.~Petz}, \bibinfo{author}{G.~Toth},
\newblock \bibinfo{title}{The {Bogoliubov} inner product in quantum
  statistics},
\newblock \bibinfo{journal}{Letters in Mathematical Physics}
  \bibinfo{volume}{27} (\bibinfo{year}{1993}) \bibinfo{pages}{205--216}.
\bibitem[{Michor et~al.(2000)Michor, Petz, and Andai}]{Michor00}
\bibinfo{author}{P.~W. Michor}, \bibinfo{author}{D.~Petz},
  \bibinfo{author}{A.~Andai},
\newblock \bibinfo{title}{{The Curvature of the Bogoliubov-Kubo-Mori Scalar
  Product on Matrices}},
\newblock \bibinfo{journal}{{Infinite Dimensional Analysis, Quantum Probability
  and Related Topics}} \bibinfo{volume}{3} (\bibinfo{year}{2000})
  \bibinfo{pages}{1--14}.
\bibitem[{Lin(2019)}]{Lin19}
\bibinfo{author}{Z.~Lin},
\newblock \bibinfo{title}{Riemannian {Geometry} of {Symmetric} {Positive}
  {Definite} {Matrices} via {Cholesky} {Decomposition}},
\newblock \bibinfo{journal}{SIAM Journal on Matrix Analysis and Applications}
  \bibinfo{volume}{40} (\bibinfo{year}{2019}) \bibinfo{pages}{1353--1370}.
\bibitem[{Hiai and Petz(2009)}]{Hiai09}
\bibinfo{author}{F.~Hiai}, \bibinfo{author}{D.~Petz},
\newblock \bibinfo{title}{Riemannian metrics on positive definite matrices
  related to means},
\newblock \bibinfo{journal}{Linear Algebra and its Applications}
  \bibinfo{volume}{430} (\bibinfo{year}{2009}) \bibinfo{pages}{3105--3130}.
\bibitem[{Hiai and Petz(2012)}]{Hiai12}
\bibinfo{author}{F.~Hiai}, \bibinfo{author}{D.~Petz},
\newblock \bibinfo{title}{Riemannian metrics on positive definite matrices
  related to means. {II}},
\newblock \bibinfo{journal}{Linear Algebra and its Applications}
  \bibinfo{volume}{436} (\bibinfo{year}{2012}) \bibinfo{pages}{2117--2136}.
\bibitem[{Dryden et~al.(2010)Dryden, Pennec, and Peyrat}]{Dryden10}
\bibinfo{author}{I.~L. Dryden}, \bibinfo{author}{X.~Pennec},
  \bibinfo{author}{J.-M. Peyrat}, \bibinfo{title}{{Power Euclidean metrics for
  covariance matrices with application to diffusion tensor imaging}},
  \bibinfo{year}{2010}. \bibinfo{note}{ArXiv e-prints}.
\bibitem[{Ha~Quang(2019)}]{HaQuang19}
\bibinfo{author}{M.~Ha~Quang},
\newblock \bibinfo{title}{{A Unified Formulation for the Bures-Wasserstein and
  Log-Euclidean/Log-Hilbert-Schmidt Distances between Positive Definite
  Operators}},
\newblock in: \bibinfo{booktitle}{Proceedings of GSI 2019 - 4th conference on
  Geometric Science of Information}, volume \bibinfo{volume}{11712} of
  \textit{\bibinfo{series}{{Lecture Notes in Computer Science}}},
  \bibinfo{publisher}{Springer International Publishing},
  \bibinfo{address}{{Toulouse, France}}, \bibinfo{year}{2019}, pp.
  \bibinfo{pages}{475--483}.
\bibitem[{Thanwerdas and Pennec(2019{\natexlab{a}})}]{Thanwerdas19-affine}
\bibinfo{author}{Y.~Thanwerdas}, \bibinfo{author}{X.~Pennec},
\newblock \bibinfo{title}{{Is affine-invariance well defined on SPD matrices? A
  principled continuum of metrics}},
\newblock in: \bibinfo{booktitle}{Proceedings of GSI 2019 - 4th conference on
  Geometric Science of Information}, volume \bibinfo{volume}{11712} of
  \textit{\bibinfo{series}{{Lecture Notes in Computer Science}}},
  \bibinfo{publisher}{Springer International Publishing},
  \bibinfo{address}{{Toulouse, France}}, \bibinfo{year}{2019}{\natexlab{a}},
  pp. \bibinfo{pages}{502--510}.
\bibitem[{Thanwerdas and Pennec(2019{\natexlab{b}})}]{Thanwerdas19-balanced}
\bibinfo{author}{Y.~Thanwerdas}, \bibinfo{author}{X.~Pennec},
\newblock \bibinfo{title}{{Exploration of Balanced Metrics on Symmetric
  Positive Definite Matrices}},
\newblock in: \bibinfo{booktitle}{Proceedings of GSI 2019 - 4th conference on
  Geometric Science of Information}, volume \bibinfo{volume}{11712} of
  \textit{\bibinfo{series}{{Lecture Notes in Computer Science}}},
  \bibinfo{publisher}{Springer International Publishing},
  \bibinfo{address}{{Toulouse, France}}, \bibinfo{year}{2019}{\natexlab{b}},
  pp. \bibinfo{pages}{484--493}.
\bibitem[{Thanwerdas and Pennec(2021)}]{Thanwerdas21-LAA}
\bibinfo{author}{Y.~Thanwerdas}, \bibinfo{author}{X.~Pennec},
  \bibinfo{title}{{O(n)-invariant Riemannian metrics on SPD matrices}},
  \bibinfo{year}{2021}. \bibinfo{note}{ArXiv e-prints}.
\bibitem[{Bhatia(1997)}]{Bhatia97}
\bibinfo{author}{R.~Bhatia}, \bibinfo{title}{Matrix {Analysis}}, volume
  \bibinfo{volume}{169} of \textit{\bibinfo{series}{Graduate {Texts} in
  {Mathematics}}}, \bibinfo{publisher}{Springer New York},
  \bibinfo{address}{New York, NY}, \bibinfo{year}{1997}.
\bibitem[{Lenglet et~al.(2004)Lenglet, Deriche, and Faugeras}]{Lenglet04}
\bibinfo{author}{C.~Lenglet}, \bibinfo{author}{R.~Deriche},
  \bibinfo{author}{O.~Faugeras},
\newblock \bibinfo{title}{Inferring {White} {Matter} {Geometry} from
  {Diffusion} {Tensor} {MRI}: {Application} to {Connectivity} {Mapping}},
\newblock in: \bibinfo{editor}{T.~Pajdla}, \bibinfo{editor}{J.~Matas} (Eds.),
  \bibinfo{booktitle}{Computer {Vision} - {ECCV} 2004}, volume
  \bibinfo{volume}{3024} of \textit{\bibinfo{series}{Lecture Notes in Computer
  Science}}, \bibinfo{publisher}{Springer}, \bibinfo{address}{Berlin,
  Heidelberg}, \bibinfo{year}{2004}, pp. \bibinfo{pages}{127--140}.
\bibitem[{Fuster et~al.(2016)Fuster, Dela~Haije, Tristán-Vega, Plantinga,
  Westin, and Florack}]{Fuster16}
\bibinfo{author}{A.~Fuster}, \bibinfo{author}{T.~Dela~Haije},
  \bibinfo{author}{A.~Tristán-Vega}, \bibinfo{author}{B.~Plantinga},
  \bibinfo{author}{C.-F. Westin}, \bibinfo{author}{L.~Florack},
\newblock \bibinfo{title}{Adjugate {Diffusion} {Tensors} for {Geodesic}
  {Tractography} in {White} {Matter}},
\newblock \bibinfo{journal}{Journal of Mathematical Imaging and Vision}
  \bibinfo{volume}{54} (\bibinfo{year}{2016}) \bibinfo{pages}{1--14}.
\bibitem[{Banerjee et~al.(2005)Banerjee, Merugu, Dhillon, and
  Ghosh}]{Banerjee05}
\bibinfo{author}{A.~Banerjee}, \bibinfo{author}{S.~Merugu},
  \bibinfo{author}{I.~S. Dhillon}, \bibinfo{author}{J.~Ghosh},
\newblock \bibinfo{title}{Clustering with {Bregman} {Divergences}},
\newblock \bibinfo{journal}{Journal of Machine Learning Research}
  \bibinfo{volume}{6} (\bibinfo{year}{2005}) \bibinfo{pages}{1705--1749}.
\bibitem[{Nielsen and Nock(2009)}]{Nielsen09}
\bibinfo{author}{F.~Nielsen}, \bibinfo{author}{R.~Nock},
\newblock \bibinfo{title}{Sided and {Symmetrized} {Bregman} {Centroids}},
\newblock \bibinfo{journal}{IEEE Transactions on Information Theory}
  \bibinfo{volume}{55} (\bibinfo{year}{2009}) \bibinfo{pages}{2882--2904}.
  \bibinfo{note}{Conference Name: IEEE Transactions on Information Theory}.
\bibitem[{Shima and Yagi(1997)}]{Shima97}
\bibinfo{author}{H.~Shima}, \bibinfo{author}{K.~Yagi},
\newblock \bibinfo{title}{Geometry of {Hessian} manifolds},
\newblock \bibinfo{journal}{Differential Geometry and its Applications}
  \bibinfo{volume}{7} (\bibinfo{year}{1997}) \bibinfo{pages}{277--290}.
\bibitem[{Amari and Armstrong(2014)}]{Amari14-DGA}
\bibinfo{author}{S.~Amari}, \bibinfo{author}{J.~Armstrong},
\newblock \bibinfo{title}{Curvature of {Hessian} manifolds},
\newblock \bibinfo{journal}{Differential Geometry and its Applications}
  \bibinfo{volume}{33} (\bibinfo{year}{2014}) \bibinfo{pages}{1--12}.
\bibitem[{Amari(2016)}]{Amari16}
\bibinfo{author}{S.~Amari}, \bibinfo{title}{Information Geometry and Its
  Applications}, \bibinfo{edition}{1st} ed., \bibinfo{publisher}{Springer
  Publishing Company, Incorporated}, \bibinfo{year}{2016}.
\bibitem[{Amari(2014)}]{Amari14}
\bibinfo{author}{S.~Amari},
\newblock \bibinfo{title}{Information geometry of positive measures and
  positive-definite matrices: Decomposable dually flat structure},
\newblock \bibinfo{journal}{Entropy} \bibinfo{volume}{16}
  (\bibinfo{year}{2014}) \bibinfo{pages}{2131--2145}.
\bibitem[{Cichocki et~al.(2011)Cichocki, Cruces, and Amari}]{Cichocki11}
\bibinfo{author}{A.~Cichocki}, \bibinfo{author}{S.~Cruces},
  \bibinfo{author}{S.~Amari},
\newblock \bibinfo{title}{Generalized alpha-beta divergences and their
  application to robust nonnegative matrix factorization},
\newblock \bibinfo{journal}{Entropy} \bibinfo{volume}{13}
  (\bibinfo{year}{2011}) \bibinfo{pages}{134--170}.
\bibitem[{Cichocki and Amari(2010)}]{Cichocki10}
\bibinfo{author}{A.~Cichocki}, \bibinfo{author}{S.~Amari},
\newblock \bibinfo{title}{Families of {Alpha}- {Beta}- and {Gamma}-
  {Divergences}: {Flexible} and {Robust} {Measures} of {Similarities}},
\newblock \bibinfo{journal}{Entropy} \bibinfo{volume}{12}
  (\bibinfo{year}{2010}) \bibinfo{pages}{1532--1568}. \bibinfo{note}{Publisher:
  Molecular Diversity Preservation International}.
\bibitem[{Pennec et~al.(2020)Pennec, Sommer, and Fletcher}]{Pennec20}
\bibinfo{author}{X.~Pennec}, \bibinfo{author}{S.~Sommer},
  \bibinfo{author}{T.~Fletcher}, \bibinfo{title}{{Riemannian Geometric
  Statistics in Medical Image Analysis}}, \bibinfo{publisher}{{Elsevier}},
  \bibinfo{year}{2020}.

\end{thebibliography}

\end{document}